\documentclass[12pt]{article}
\usepackage[margin=2.5cm]{geometry}
\usepackage{authblk}

\usepackage{amsmath}
\usepackage{amssymb}
\usepackage{amsthm}

\usepackage{enumitem}
\usepackage{hyperref}
\usepackage{bbm}
\usepackage{dsfont}
\usepackage{setspace}

\hypersetup{
	colorlinks   = true,
	linkcolor=blue,
	citecolor=blue
}

\newtheorem{theorem}{Theorem}[section]
\newtheorem{lemma}[theorem]{Lemma}
\newtheorem{corollary}[theorem]{Corollary}
\newtheorem{remark}[theorem]{Remark}

\newcommand{\bfX}{\mathbf{X}}
\newcommand{\bfY}{\mathbf{Y}}

\newcommand{\NN}{\mathbb{N}}
\newcommand{\RR}{\mathbb{R}}
\newcommand{\EE}{\mathbb{E}}
\newcommand{\bbone}{\mathbbm{1}}

\newcommand{\rE}{\mathcal{E}}
\newcommand{\rP}{\mathcal{P}}

\newcommand{\vp}{\varphi}
\newcommand{\bX}{\overline{X}}
\newcommand{\bY}{\overline{Y}}
\newcommand{\conx}{\mathbb{X}}
\newcommand{\cony}{\mathbb{Y}}
\newcommand{\dcbo}{\rho}
\newcommand{\dmf}{\overline{\rho}}
\newcommand{\dpmf}{\overline{\mu}}

\DeclareMathOperator{\law}{Law}
\newcommand{\diff}{\,\mathrm{d}}

\author[1]{Hui Huang}
\author[2]{Jethro Warnett}
\affil[1]{School of Mathematics, Hunan University,  Changsha, 410082, China\\
Email: huihuang1@hnu.edu.cn\vspace{.3cm}}
\affil[2]{Mathematical Institute, University of Oxford, Woodstock Road, Oxford, OX2 6GG, United Kingdom,
Email: warnett@maths.ox.ac.uk}

\begin{document}
    \title{Well-posedness and mean-field limit estimate of a consensus-based algorithm for min-max problems}
    \maketitle
    
    \textbf{Abstract.} The recent work \cite{borghi_huang_qiu_2024_particle} proposes a derivative-free consensus-based particle method that computes global solutions to nonconvex–nonconcave min–max problems and establishes global exponential convergence in the sense of the mean-field law.
    This paper aims to address the theoretical gaps in \cite{borghi_huang_qiu_2024_particle}, specifically by providing a quantitative estimate of the mean-field limit with respect to the number of particles, as well as establishing the well-posedness of both the finite particle model and the corresponding mean-field dynamics.
    \allowdisplaybreaks

    \section{Introduction}

    Min-max problems are a central class of optimization problems. They describe two competing agents, where one tries to minimize an objective function, whilst the other tries to maximize it. The min-max problem was first introduced in Nash's seminal paper \cite{nash_1950_equilibrium}, in which he established conditions for competing agents to reach an equilibria. The min-max problem has also found substantial applications in machine learning. Specifically, they have become a central research question for Generative Adversarial Networks (GAN) \cite{goodfellow_pougetabadie_mirza_xu_wardefarley_ozair_courville_bengio_2014_generative}. Here, two neural networks compete with each other, the first tries to generate data, whilst the other tries to discriminate between the generated data and the real-world data. For a broader perspective of the history and application of the min-max problem, we refer the reader to \cite{zhang_poupart_yu_2022_optimality} and the references therein.

    Recently, a consensus-based optimization (CBO) algorithm has been proposed to find the global optimisers of nonconvex-nonconcave min-max problems \cite{borghi_huang_qiu_2024_particle}. The CBO algorithm has been originally proposed in \cite{pinnau_totzeck_tse_martin_2017_consensusbased, carrillo_choi_totzeck_tse_2018_analytical}, which is  inspired by collective behavior in nature, such as flocking birds or swarming fish. It is used to solve optimization problems by simulating the interaction of particles that collectively and stochastically move toward an optimal solution. The agents share information and adjust their positions based on a consensus mechanism, which drives the system toward a global optimum. The CBO algorithm offers several advantages, including its derivative-free nature and amenability to mathematical analysis. It has since been extended and adapted to a wide range of optimization settings, including such as saddle-point problems \cite{huang_qiu_riedl_2024_consensusbased}, multi-level optimization \cite{herty_huang_kalise_kouhkouh_2025_multiscale,trillos_kumarakash_li_riedl_zhu_2025_defending}, multi-objective optimization \cite{borghi_herty_pareschi_2023_adaptive}, constrained optimization \cite{beddrich_chenchene_fornasier_huang_wohlmuth_2024_constrained,bungert_hoffmann_kim_roith_2025_mirrorcbo}, and multi-player games \cite{chenchene_huang_qiu_2025_consensusbased}.  Rather than attempting to include a necessarily incomplete
account of this very fast-growing field, we refer the interested reader to the review papers \cite{totzeck_2021_trends} and \cite{wang_li_huang_2025_mathematical} for a more recent and
relatively comprehensive report.

    The authors in \cite{borghi_huang_qiu_2024_particle} consider a two-agent CBO algorithm. We are given a cost function $\rE:\RR^{d_1}\times\RR^{d_2}\to \RR$. The first agent picks a decision, aiming to make the objective as small as possible under the second agent’s maximizing response. This means
    \begin{align}
        \label{eq:min_max}\min_{x\in\RR^{d_1}}\max_{y\in\RR^{d_2}}\;\rE(x,y).
        \tag{Min-Max}
    \end{align}
    In order to define a global optimal solution to this problem, we introduce the definition of the upper envelope function
    \begin{align*}
        \bar{\rE}(x):=\sup_{y\in \RR^{d_2}}\rE(x,y)
        \qquad\forall x\in\RR^{d_1},
    \end{align*}
    where this function may take the value $+\infty$. Now, we say that a point $(x^\ast,y^\ast)\in \RR^{d_1}\times\RR^{d_2}$ is a gobal min-max solution to \eqref{eq:min_max} if
    \begin{align*}
        x^\ast\in\arg\min_{x\in\RR^{d_1}} \bar{\rE}(x)
        \qquad
        \text{ and }
        \qquad y^\ast\in\arg\max_{y\in\RR^{d_2}}\rE(x^\ast,y).
    \end{align*}
    
    Each agent is attributed a collection of $N$ particles, which are denoted by the vectors $X^1,\ldots, X^N\in\RR^{d_1}$ for the first one and $Y^1,\ldots, Y^N\in\RR^{d_2}$ for the second one. Both agents have their respective consensus point that they drift towards, namely
    \small{\begin{align}\label{eq:defi_consensus}
        \begin{split}
            \conx_{\alpha,\beta}(\dcbo_t^{X,N},\dcbo_t^{Y,N}) &:= \frac{\displaystyle\int_{\mathbb{R}^d} x\;\omega^{\rE}_{\alpha}(x,\cony_{\beta}(\dcbo_t^{Y,N},x))\diff \dcbo_t^{X,N}(x)}{\displaystyle\int_{\mathbb{R}^d} \omega^{\rE}_{\alpha}(x,\cony_{\beta}(\dcbo_t^{Y,N},x))\diff \dcbo_t^{X,N}(x)},\\ 
        \cony_{\beta}(\dcbo_t^{Y,N},X_t^i) &:= \frac{\displaystyle\int_{\mathbb{R}^d} y\,\omega^{\rE}_{-\beta}(X_t^i,y)\diff\dcbo^{Y,N}_t( y)}{\displaystyle\int_{\mathbb{R}^d} \omega^{\rE}_{-\beta}(X_t^i,y)\diff\dcbo^{Y,N}_t(y)},
        \end{split}
    \end{align}}
    where for any value $\gamma\in\RR$ we define the weights and the empirical approximations by
    \begin{align*}
        \omega_\gamma^\rE(x,y):=e^{-\gamma \rE(x,y)},\quad
        \dcbo_t^{X,N} := \frac{1}{N}\sum_{i=1}^N \delta_{X_t^i},\quad
        \dcbo_t^{Y,N} := \frac{1}{N}\sum_{i=1}^N \delta_{Y_t^i}.
    \end{align*}
    The continuous CBO dynamics is expressed by an interacting particle evolution system represented by the following stochastic differential equations, for every index $i \in [N]$,
    \begin{align}
        \label{eq:cbo}
        \begin{cases}
            \begin{aligned}
                \diff X_t^i =& -\lambda(X_t^i - \conx_{\alpha,\beta}(\dcbo_t^{X,N},\dcbo_t^{Y,N}))\diff t+ \sigma D(X_t^i - \conx_{\alpha,\beta}(\dcbo_t^{X,N},\dcbo_t^{Y,N}))\diff B_t^{X,i}, 
            \end{aligned}\\
            \begin{aligned}
                \diff Y_t^i =& -\lambda(Y_t^i - \cony_{\beta}(\dcbo_t^{Y,N}, X_t^i))\diff t+ \sigma D(Y_t^i - \cony_{\beta}(\dcbo_t^{Y,N}, X_t^i))\diff B_t^{Y,i},
            \end{aligned}  
        \end{cases}
        \tag{CBO}
    \end{align}
    where $\lambda,\sigma>0$ are drift and diffusion parameters respectively and $\{B^{m,i}\}_{m\in \{X,Y\}, i\in [N]}$ are independent Brownian motions. Here $D:\mathbb{R}^d \to \mathbb{R}^{d\times d}$ is a (possibly agent-dependent) matrix-valued map that is Lipschitz continuous with respect to the Frobenius norm. For simplicity, we take either the anisotropic choice $D(x)=\mathrm{diag}(x_1,\ldots,x_d)$ or the isotropic choice $D(x)=|x|\,\mathrm{id}$, for $x\in\mathbb{R}^d$. The initial law satisfies
    \begin{align*}
        \law(X_0^1,Y_0^1, \dots, X_0^N,Y_0^N) = (\dcbo_0^X\otimes\dcbo_0^Y)^{\otimes N},
    \end{align*}
     given initial probability distributions $\dcbo_0^X$ on $\RR^{d_1}$ and $\dcbo_0^Y$ on $\RR^{d_2}$. The paper \cite{borghi_huang_qiu_2024_particle} formally states that the mean-field limit of the above particle dynamics is dictated by the following stochastic differential equation
    \begin{align}
        \label{eq:mf_cbo}
        \begin{cases}
            \begin{aligned}
                \diff\bX_t =& -\lambda(\bX_t - \conx_{\alpha,\beta}(\dmf_t^X, \dmf_t^Y))\diff t+ \sigma D(\bX_t - \conx_{\alpha,\beta}(\dmf_t^X, \dmf_t^Y))\diff B_t^X, 
            \end{aligned}\\
            \begin{aligned}
                \diff \bY_t = &-\lambda(\bY_t - \cony_{\beta}(\dmf_t^Y, \bX_t))\diff t+ \sigma D(\bY_t - \cony_{\beta}(\dmf_t^Y, \bX_t))\diff B_t^Y,
            \end{aligned}
        \end{cases}
        \tag{MF CBO}
    \end{align}
    where we have the law of the two agents 
    \begin{align*}
    	\dmf_t^X := \law\big(\,\bX_t\big),\qquad
        \dmf_t^Y := \law\big(\,\bY_t\big).
    \end{align*}
    The initial law satisfies $
        \law(\bX_0,\bY_0) = \dmf_0^X\otimes\dmf_0^Y,$ 
     given initial probability distributions $\dmf_0^X$ on $\RR^{d_1}$ and $\dmf_0^Y$ on $\RR^{d_2}$.

    \subsection{Motivation}
    The authors \cite{borghi_huang_qiu_2024_particle} rigorously prove global convergence in the large particle limit. Specifically, they establish the convergence of the mean-field dynamics $(\bX_t, \bY_t)$ in \eqref{eq:mf_cbo} to the global min-max solution $(x^\ast,y^\ast)$ as $t$ approaches infinity. To achieve this, they introduce the error functions:
    \begin{align}\label{eq:defi_V}
        \begin{split}
            &V^X(t): = \mathbb{E}|\bX_t - x^\ast|^2,\quad V^Y(t): = \mathbb{E}|\bY_t - y^\ast|^2,\\
        &V(t): = V^X(t)+V^Y(t).
        \end{split}
    \end{align}
    By analyzing the decay behavior of the (cumulative) error function $V(t)$, they demonstrate that $V(t)$ decays exponentially at a finite time $T_* > 0$, with a decay rate that can be controlled through the parameters of the CBO method. Specifically, it holds that
    \begin{align*}
        V(t) \leq V(0) \exp\left(-\frac{2\lambda - \sigma^2}{2} t\right)
    \end{align*}
    for $t \in [0, T_*]$, and $V(T_*) \leq \varepsilon$ for any given accuracy $\varepsilon > 0$. It is worth noting that one can also establish particle-level convergence of CBO methods directly, using a variance-based analysis; see, e.g., \cite{byeon_ha_hwang_ko_yoon_2025_consensus,ko_ha_jin_kim_2022_convergence,ha_hwang_kim_2024_timediscrete}.
    
    However, the justification for passing to the mean-field limit from the particle system \eqref{eq:cbo} to the mean-field dynamics \eqref{eq:mf_cbo} remains only formal, and a rigorous mathematical analysis is  absent in \cite{borghi_huang_qiu_2024_particle}. Establishing a rigorous proof of the CBO models has been a challenging task, primarily due to the fact that the consensus point defined in \eqref{eq:defi_consensus} is only locally Lipschitz.
    
   Recent years have seen substantial progress on the mean-field limit analysis for CBO. For example, \cite{fornasier_huang_pareschi_sunnen_2020_consensusbased,ha_kang_kim_kim_yang_2022_stochastic} derived mean-field limit estimates for CBO variants on compact manifolds by assuming a globally Lipschitz consensus point. Later, \cite{huang_qiu_2022_meanfield} proved the mean-field limit for the standard CBO model via a compactness argument based on Prokhorov’s theorem; however, this method did not yield an explicit convergence rate in terms of the particle number $N$. Subsequent work \cite{fornasier_klock_riedl_2024_consensusbased,huang_qiu_riedl_2023_global} established a probabilistic mean-field approximation, showing that the limit estimate holds with high probability. The high-probability restriction was removed in \cite{gerber_hoffmann_vaes_2025_meanfield} through an improved stability bound for the consensus point, which was further extended to the multi-species setting in \cite{huang_warnett_2025_wellposedness}.
    Finally, uniform-in-time type of mean-field estimates have been recently obtained \cite{huang_kouhkouh_2025_uniformintime,gerber_hoffmann_kim_vaes_2025_uniformintime,bayraktar_ekren_zhou_2025_uniformintime,choi2025modified}.

    The primary objective of this paper is to address some theoretical gaps in \cite{borghi_huang_qiu_2024_particle} by providing a quantitative estimate of the mean-field limit with respect to the number $N$ of particles. Additionally, we establish the well-posedness of both the finite particle model \eqref{eq:cbo} and the mean-field dynamics \eqref{eq:mf_cbo}. We extend the results of \cite{gerber_hoffmann_vaes_2025_meanfield} to settings where the dynamics depend on several distributions with inter-agent dependencies, rather than a single distribution.
    
    For this reason we cannot take this extension for granted, as it is not clear that the same mean-field argument still holds and if the convergence rate remains unchanged. This extension requires several non-trivial steps, such as the measure--variable stability result in Lemma \ref{lem:stab_est_y}, the coupled stability result in Lemma \ref{lemma:basic_stability_estimate}, the weighted moment bound in Lemma \ref{lemma:moment_estimates_for_the_empirical_measures}, the convergence established in Lemma \ref{lemma:convergence_of_the_weighted_mean_for_iid_samples_y} that is uniform with respect to the state of the agent $X$, or the coupled convergence in Lemma \ref{lemma:convergence_of_the_weighted_mean_for_iid_samples}. Our proof for the well-posedness of \eqref{eq:mf_cbo} (See Section \ref{section:meanfield_limit_for_cbo}) is fundamentally different due to the agent--dependent nature of the consensus function $\cony_{\beta}(\dmf_t^Y, \bX_t)$ for agent $Y$. Consequently, the results presented in \cite{carrillo_choi_totzeck_tse_2018_analytical} cannot be directly applied.  Finally, in the proof of Theorem \ref{theorem:meanfield_of_cbo} for the mean-field limit, we need to take the mean-field limit in two measures simultaneously, and the mean-field limit of the agent $Y$ must happen uniformly among all possible states of the agent $X$. Our approach therefore both leverages and advances the methodology of \cite{gerber_hoffmann_vaes_2025_meanfield}.
    
    \subsection{Main Results}
    At first we ferment the notation that we use throughout the paper. Given any radius $R>0$, we define the closed ball of radius $R$ centered at the origin in $\RR^{d_2}$ by $B_R$. We denote by $|x|$ and $|A|$ the absolute value of a vector $x$ the Lebesgue measure of a Borel measurable Euclidean subset $A$. For any positive values $1<p<\infty$ and $R>0$ we let $\rP_{p,R}(\RR^d)$ (or $\rP_p(\RR^d)$) denote the space of probability measures with $p$-moment bounded by $R$ (or finite $p$-moment), that means
    \begin{align*}
        \int_{\RR^d}|x|^p\diff\mu(x)\leq R,\quad
        \int_{\RR^d}|x|^p\diff\nu(x)<\infty,
        \quad\forall \mu\in \rP_{p,R}(\RR^d) \quad \forall \nu\in \rP_p(\RR^d).
    \end{align*}
    The Wasserstein distance between two probability measures $\mu,\nu\in\rP_p(\RR^d)$ is the minimal transportation cost required to move one distribution to the other, given by
    \begin{align*}
        W_p(\mu,\nu)
        = \left( \inf
        \int_{\RR^d\times\RR^d} |x-y|^p \diff\gamma(x,y)
      \right)^{1/p},
    \end{align*}
    where we take the infimum over all $\gamma\in \rP(\RR^d\times\RR^d)$ satisfying the constraint $\pi(A\times \RR^d)=\mu(A)$ and $\pi(\RR^d\times A)=\nu(A)$ for all Borel measurable $A\subseteq \RR^d$. We denote the set of all admissible $\gamma$, called a coupling, by $\Gamma(\mu,\nu)$. We define the expected value of a probability measure $\mu\in\rP_p(\RR^d)$ by
    \begin{align*}
        \EE[\mu]:=\int_{\RR^d} x\diff \mu(x).
    \end{align*}
    For any random vector $X$ and value $p\geq 1$ we write $\EE|X|^p$ as a shorthand for $\EE\big[|X|^p\big]$. Finally, for any vector $x\in\RR^d$ we let $\delta_x$ denote the Dirac measure at $x$.
    
    Throughout this paper we assume that the cost function is bounded and that the Lipschitz constant has polynomial growth.
    \begin{enumerate}[label=(A\arabic*)]
        \item \label{ass:bnd} There exists a constant $c>0$ such that we have boundedness
    		\begin{align*}
    			-c
    			\leq \rE(x,y)
    			\leq c.
    		\end{align*}
        \item \label{ass:lip} There exists constants $C>0$ and $s\geq 0$ such that for $z_0,z_1\in\RR^{d_1+d_2}$, where we set the vector $z_k=(x_k,y_k)\in \RR^{d_1}\times \RR^{d_2}$ for $k\in \{0,1\}$, we have the polynomial-growth Lipschitz property
    		\begin{align*}
    			|\rE(z_0)-\rE(z_1)|
    			\leq C(1+|z_0|+|z_1|)^s\, |z_0-z_1|.
    		\end{align*}
    \end{enumerate}
 We develop a novel existence theory for the nonlinear coupled setting by significantly adapting the strategy of \cite[Theorems 2.2–2.3]{gerber_hoffmann_vaes_2025_meanfield}, overcoming key challenges arising from the coupling that are absent in the original uncoupled framework.

    \begin{theorem}[{Existence and uniqueness for \eqref{eq:cbo}}]
    	\label{theorem:existence_and_uniqueness_for_cbo}
        Let Assumptions \ref{ass:lip} hold. Then the SDEs \eqref{eq:cbo} posses unique strong solutions $\{(X_t^i, Y_t^i)\}_{i\in [N]}$ for any initial conditions $\{(X_0^i,Y_0^i)\}_{i\in [N]}$ that are independent of the Brownian motions $\{(B_t^{X,i}, B_t^{Y,i})\}_{i\in [N]}$. The solutions are almost surely continuous.
    \end{theorem}
    
    \begin{theorem}[{Existence and uniqueness for \eqref{eq:mf_cbo}}]
    	\label{theorem:existence_and_uniqueness_for_mfcbo}
        Let Assumptions \ref{ass:bnd} and \ref{ass:lip} hold, let $p>2$ if $s=0$, otherwise let $p\geq 2+s$. Then, for all $T > 0$, $\dmf_0^X\in\rP_p(\RR^{d_1})$ and $\dmf_0^Y\in\rP_{p+\eta}(\RR^{d_2})$, there exists unique processes $\bX:\Omega\to C^0([0,T], \RR^{d_1})$ and $\bY:\Omega\to C^0([0,T], \RR^{d_2})$ satisfying \eqref{eq:mf_cbo} in the strong sense with initial condition $\law(\bX_0,\bY_0)= \dmf_0^X\otimes \dmf_0^Y$. Furthermore, we have the bounds
        \begin{align*}
            \begin{split}
                &\EE \left[ 
                \sup_{\substack{t \in [0, T]}} 
                \;|\bX_t|^p+|\bY_t|^p+|\conx_{\alpha,\beta}(\dmf_t^X, \dmf_t^Y)|+|\cony_\beta(\dmf_t^Y, \bX_t)|
                \right] < \infty,\\
                &\dmf_t^X = \law\big(\,\bX_t\big),\quad
                \dmf_t^Y = \law\big(\,\bY_t\big)
            \end{split}
        \end{align*}
        and the function $t \mapsto (\conx_{\alpha,\beta}(\dmf_t^X, \dmf_t^Y),\cony_\beta(\dmf_t^Y, \bX_t))$ is continuous over $[0, T]$.
    \end{theorem}
   Furthermore, we also close a key gap in \cite{borghi_huang_qiu_2024_particle}, by providing the first quantitative convergence rate for the mean-field limit.
    
    \begin{theorem}[{Mean-field limit of \eqref{eq:cbo}}]
        \label{theorem:meanfield_of_cbo}
    	Let Assumptions \ref{ass:bnd} and \ref{ass:lip} hold with $q\geq 2(2+s)$, $p\in (0,\frac{q}{2}]$, $\rho_0^X\in \rP_q(\RR^{d_1})$ and $\rho_0^Y\in \rP_q(\RR^{d_2})$. We assume the particles $\{(X^i,Y^i)\}_{i\in[N]}$ satisfy \eqref{eq:cbo} and $\{(\bX^i,\bY^i)\}_{i\in[N]}$ are $N$ i.i.d. samples for each player from \eqref{eq:mf_cbo}. They both use the same standard Brownian motions $\{(B^{X,i}, B^{Y,i})\}_{i\in[N]}$, with the same initial condition $\law(X_0^i,Y_0^i)=\law(\bX_0^i,\bY_0^i)=\rho_0^X\otimes \rho_0^Y$. Then, for each time $T>0$, there exists a positive constant $C>0$ independent of $N$ such that
        \begin{align*}
        \sup_{i\in [N]}
    		\;\left(\EE\left[\sup_{t\in [0,T]}\;\big\vert X_t^i- \bX_t^i \big\vert^p+\big\vert Y_t^i- \bY_t^i \big\vert^p\right]\right)^{\frac{1}{p}}\leq C N^{-\gamma},
    	\end{align*}
        where we define the exponent
        \begin{align*}
            \gamma:=\min\left\{\frac{1}{2},\; \frac{q-p}{2p^2},\; \frac{q-(2+s)}{2 (2+s)^2}\right\}.
        \end{align*}
    \end{theorem}

    \begin{remark}
        We achieve Monte-Carlo convergence rate, this means $\gamma=\frac{1}{2}$ and $p=2$, whenever $q\geq 6\vee \big((2+s)+(2+s)^2\big)$.
    \end{remark}

    \begin{remark}
        The results presented in this paper are closely related to another recent publication \cite{huang_warnett_2025_wellposedness}. Due to substantial differences in proof strategies, we decided to present these results separately. Our decision is purely based on methodological clarity and novelty, aiming to maintain transparency and avoid any unintended misrepresentation.
    \end{remark}
    
    \subsection{Paper structure}
    \label{section:particle_stab_weighted_mean}
    The paper will be subdivided into four sections. First in section \ref{section:necessary_lemmas} we state several necessary lemmas for the paper. Then, in section \ref{section:wellposedness_for_cbo}, \ref{section:wellposedness_for_mfcbo} and \ref{section:meanfield_limit_for_cbo} we prove the Theorems \ref{theorem:existence_and_uniqueness_for_cbo}, \ref{theorem:existence_and_uniqueness_for_mfcbo} and \ref{theorem:meanfield_of_cbo} respectively.

    \section{Necessary Lemmas}
    \label{section:necessary_lemmas}

    In this section we adapt several results from \cite{carrillo_choi_totzeck_tse_2018_analytical, gerber_hoffmann_vaes_2025_meanfield} to the agent-dependent setting. We will use these to prove Theorems \ref{theorem:existence_and_uniqueness_for_cbo}, \ref{theorem:existence_and_uniqueness_for_mfcbo} and \ref{theorem:meanfield_of_cbo}. We state a Wasserstein stability result of the consensus in section \ref{subsection:wasserstein:stability_estimate_for_weighted_mean}, we control the moment of the solution to \eqref{eq:cbo} and \eqref{eq:mf_cbo} in section \ref{section:moment_estimate_for_cbo_dynamics} and finally we derive a law-of-large-number result in section \ref{section:meanfield_limit_of_consensus_for_mfcbo}.

    \subsection{Wasserstein stability estimate for weighted mean}
    \label{subsection:wasserstein:stability_estimate_for_weighted_mean}

    \begin{lemma}[Estimate on weighted mean]
        \label{lem:est_weighted_mean}
        Let \ref{ass:bnd} hold and let $p\in [1,\infty)$. Then, for every $\mu^X\in\rP_p(\RR^{d_1})$, $\mu^Y\in\rP_p(\RR^{d_2})$ and $U\in \RR^{d_1}$ there exists constants $C_0,C_1>0$ such that
        \begin{enumerate}[label=(C\arabic*)]
            \item\label{eq:lower_bnd_weight} We can bound the weights from below
        \begin{align*}
            C_0\leq \mu^X[\omega_\alpha(\cdot,\cony_\beta(\mu^Y,\cdot))]
            \quad\text{and}\quad
            C_0\leq \mu^Y[\omega_{-\beta}(U, \cdot)].
        \end{align*}
        \item\label{eq:upper_bnd_weighted_mean} We can bound the weighted means from above
        \begin{align*}
            \mu^X[x\,\omega_\alpha(\cdot,\cony_\beta(\mu^Y,\cdot))]
            &\leq C_1\,\left(\int_{\RR^{d_1}}|x|^{p}\diff \mu^X(x)\right)^{\frac{1}{p}},\\
            \mu^Y[y\,\omega_{-\beta}(U,\cdot)]
            &\leq C_1\,\left(\int_{\RR^{d_2}}|y|^p\diff \mu^Y(y)\right)^{\frac{1}{p}}.
        \end{align*}
        \item\label{eq:upper_bnd_cons} We can bound the consensus from above
        \begin{align*}
            |\conx_{\alpha,\beta}(\mu^X,\mu^Y)|
            &\leq \frac{C_1}{C_0}\left(\int_{\RR^{d_1}}|x|^p\diff \mu^X(x)\right)^{\frac{1}{p}},\\
            |\cony_{\beta}(\mu^Y,U)|
            &\leq \frac{C_1}{C_0} \left(\int_{\RR^{d_2}}|y|^p\diff \mu^Y(x)\right)^{\frac{1}{p}}.
        \end{align*}
        \end{enumerate}
    \end{lemma}
    \begin{proof}
        The lower bound \ref{eq:lower_bnd_weight} is clear by \ref{ass:bnd}. For the upper bound \ref{eq:upper_bnd_weighted_mean} we use \ref{ass:bnd} and the Jensen inequality. Finally, \ref{eq:upper_bnd_cons} is an immediate consequence of \ref{eq:lower_bnd_weight} and \ref{eq:upper_bnd_weighted_mean}. 
    \end{proof}

    \begin{lemma}[Basic stability estimate for $\cony_\beta$]
        \label{lem:stab_est_y}
        Suppose that Assumptions \ref{ass:bnd} and \ref{ass:lip} hold. Then, for all $R>0$ and $p\geq s+2$, there exists a constant $C>0$ depending on $R$ and $p$ such that the following assertions hold true:
        \begin{enumerate}[label=(B\arabic*)]
            \item\label{est:y_var_bas} For all $\mu^Y\in \rP_{p,R}(\RR^{d_2})$ and $U,V\in\RR^{d_1}$ we have
        	\begin{align*}
        		|\cony_\beta(\mu^Y, U)-\cony_\beta(\mu^Y,V)|
        		\leq C\,(1+|U|+|V|)^s\;|U-V|.
        	\end{align*}
            \item\label{est:y_msr_bas} For all $\mu^Y, \nu^Y\in \rP_{p,R}(\RR^{d_2})$ and $X\in\RR^{d_1}$ we have
        	\begin{align*}
        		|\cony_\beta(\mu^Y, X)-\cony_\beta(\nu^Y,X)|
        		\leq C\,(1+|X|)^s\;W_p(\mu^Y,\nu^Y).
        	\end{align*}
            \end{enumerate}
    \end{lemma}
    \begin{proof}
        First we prove \ref{est:y_var_bas}. Thus, we estimate
        \begin{align}
            \label{eq:stab_est_num_cons_diff_var_start}
            \begin{split}
                &|\cony_\beta(\mu^Y,U)-\cony_\beta(\mu^Y,V)|\\
            &\leq \frac{|\mu^Y [y\, \omega_{-\beta}(U,y)]-\mu^Y [y\, \omega_{-\beta}(V,y)]|}{\mu^Y [\omega_{-\beta}(U,y)]}+\mu^Y [y\, \omega_{-\beta}(V,y)]\,\frac{|\mu^Y [\omega_{-\beta}(U,y)]-\mu^Y [\omega_{-\beta}(V,y)]|}{\mu^Y [\omega_{-\beta}(U,y)]\mu^Y [\omega_{-\beta}(V,y)]}.
            \end{split}
        \end{align}
        We can bound all terms except for the differences in \eqref{eq:stab_est_num_cons_diff_var_start} by using \ref{eq:lower_bnd_weight} and \ref{eq:upper_bnd_weighted_mean} of Lemma \ref{lem:est_weighted_mean} respectively. Hence, we only need to bound the differences. To that end, we compute
        \begin{align*}
            |\mu^Y [y\, \omega_{-\beta}(U,y)]-\mu^Y [y\, \omega_{-\beta}(V,y)]|\leq \int_{\RR^{d_2}} |y|\,| \omega_{-\beta}(U,y)-\omega_{-\beta}(V,y)|\diff\mu^Y(y).
        \end{align*}
        We use \ref{ass:lip} and Jensen's inequality and find
        \begin{align*}
            &\int_{\RR^{d_2}} |y|\,| \omega_{-\beta}(U,y)-\omega_{-\beta}(V,y)|\diff\mu^Y(y)\\
            &\leq C\int_{\RR^{d_2}} |y|\,(1+|U|+|V|+|y|)^s \,|U-V|\diff\mu^Y(y)\\
            &\leq C\left(\int_{\RR^{d_2}} |y|^{s+1} \diff\mu^Y(y)+\int_{\RR^{d_2}} |y| \diff\mu^Y(y)\,(1+|U|+|V|)^s\right)\,|U-V|\\
            &\leq C\,(1+|U|+|V|)^s\;|U-V|.
        \end{align*}
        Here $C>0$ is a constant that depends on $p$ and $R$ and changes line from line. We bound the other numerator in \eqref{eq:stab_est_num_cons_diff_var_start} similarly. By combining all inequalities we see that \ref{est:y_var_bas} holds.

        Next, we prove \ref{est:y_msr_bas}. Thus, we estimate 
        \begin{align}
            \label{eq:stab_est_cons_diff_meas_start}
            \begin{split}
                &|\cony_\beta(\mu^Y,X)-\cony_\beta(\nu^Y,X)|\\
            &\leq \frac{|\mu^Y [y\, \omega_{-\beta}(X,y)]-\nu^Y [y\, \omega_{-\beta}(X,y)]|}{\mu^Y [\omega_{-\beta}(X,y)]}+\nu^Y [y\, \omega_{-\beta}(X,y)]\,\frac{|\mu^Y [\omega_{-\beta}(X,y)]-\nu^Y [\omega_{-\beta}(X,y)]|}{\mu^Y [\omega_{-\beta}(X,y)]\nu^Y [\omega_{-\beta}(X,y)]}.
            \end{split}
        \end{align}
        We can bound all the terms apart from the differences in \eqref{eq:stab_est_cons_diff_meas_start} by using \ref{eq:lower_bnd_weight} and \ref{eq:upper_bnd_weighted_mean} of Lemma \ref{lem:est_weighted_mean} respectively. Hence, we only need to bound the numerators. To that end, for any coupling $\gamma\in\Gamma(\mu^Y,\nu^Y)$ we compute
        \begin{align}
            \label{eq:stab_est_num_cons_diff_meas}
            \begin{split}
                &|\mu^Y [y\, \omega_{-\beta}(X,y)]-\nu^Y [y\, \omega_{-\beta}(X,y)]|\\
            &\leq \int_{\RR^{d_2}\times \RR^{d_2}}|u\, \omega_{-\beta}(X,u)-v\, \omega_{-\beta}(X,v)|\diff\gamma(u,v)\\
            & \leq \int_{\RR^{d_2}\times \RR^{d_2}}\omega_{-\beta}(X,u)\,|u -v|\diff\gamma(u,v)+\int_{\RR^{d_2}\times \RR^{d_2}}|v|\,|\omega_{-\beta}(X,u)-\omega_{-\beta}(X,v)|\diff\gamma(u,v).
            \end{split}
        \end{align}
        For the first term on the RHS in \eqref{eq:stab_est_num_cons_diff_meas}, we use \ref{ass:bnd} and Jensen's inequality to derive the bound
        \begin{align*}
            \int_{\RR^{d_2}\times \RR^{d_2}}\omega_{-\beta}(X,u)\,|u -v|\diff\gamma(u,v)
            \leq e^{\beta c}\left(\int_{\RR^{d_2}\times \RR^{d_2}} |u -v|^p \diff\gamma(u,v) \right)^{\frac{1}{p}}.
        \end{align*}
        For the second term on the RHS in \eqref{eq:stab_est_num_cons_diff_meas}, we use \ref{ass:lip}, the Hölder inequality, and that $\frac{(s+1)p}{p-1}\leq p$
        \begin{align*}
            &\int_{\RR^{d_2}\times \RR^{d_2}}|v|\,|\omega_{-\beta}(X,u)-\omega_{-\beta}(X,v)|\diff\gamma(u,v)\\
            & \leq C\int_{\RR^{d_2}\times \RR^{d_2}} (1+|u|+|v|)^{s+1} |u-v|\diff\gamma(u,v)+C|X|^s\int_{\RR^{d_2}\times \RR^{d_2}} |v|\,|u-v|\diff\gamma(u,v)\\
            & \leq C\left(1+|X|^s\right)\left(\int_{\RR^{d_2}\times \RR^{d_2}} |u -v|^p \diff\gamma(u,v) \right)^{\frac{1}{p}}.
        \end{align*}
        Here $C>0$ is a constant that depends on $p$ and $R$ and changes from line to line. We derive a similar estimate for the other difference in \eqref{eq:stab_est_cons_diff_meas_start}. Thus, by taking the infimum over all couplings $\gamma$, we see that \ref{est:y_msr_bas} holds.
    \end{proof}

    \begin{lemma}[Basic Stability Estimate for $\conx_{\alpha,\beta}$]
        \label{lemma:basic_stability_estimate}
        Let Assumptions \ref{ass:bnd} and \ref{ass:lip} hold. Let $R>0$ and $p\geq 2(s+1)$. Then, there exists $C>0$ depending on $R$ and $p$, such that for any pairs $\mu^X,\nu^X\in\rP_{p,R}(\RR^{d_1})$ and $\mu^Y,\nu^Y\in \rP_{p,R}(\RR^{d_2})$ the following inequality holds
        \begin{align*}
            |\conx_{\alpha,\beta}(\mu^X, \mu^Y)-\conx_{\alpha,\beta}(\nu^X, \nu^Y)|
            \leq C (W_p(\mu^X,\nu^X)+W_p(\mu^Y,\nu^Y))\,.
        \end{align*}

    \end{lemma}
    \begin{proof}
        The goal of this proof is to derive estimates to the following four terms 
        \begin{align}
            \notag
            &|\conx_{\alpha,\beta}(\mu^X,\mu^Y)-\conx_{\alpha,\beta}(\nu^X,\nu^Y)|\\
            &\leq 
            \left\vert
            \frac{\mu^X[x\,\omega_\alpha(\cdot,\cony_\beta(\mu^Y,\cdot))]}{\mu^X[\omega_\alpha(\cdot,\cony_\beta(\mu^Y,\cdot))]}
            -\frac{\mu^X[x\,\omega_\alpha(\cdot,\cony_\beta(\nu^Y,\cdot))]}{\mu^X[\omega_\alpha(\cdot,\cony_\beta(\mu^Y,\cdot))]}
            \right\vert\label{eq:conx_est_bound_one}
            \tag{I}\\[.25cm]
            &\phantom{\leq}+\left\vert
            \frac{\mu^X[x\,\omega_\alpha(\cdot,\cony_\beta(\nu^Y,\cdot))]}{\mu^X[\omega_\alpha(\cdot,\cony_\beta(\mu^Y,\cdot))]}
            -\frac{\nu^X[x\,\omega_\alpha(\cdot,\cony_\beta(\nu^Y,\cdot))]}{\mu^X[\omega_\alpha(\cdot,\cony_\beta(\mu^Y,\cdot))]}
            \right\vert\label{eq:conx_est_bound_two}
            \tag{II}\\[.25cm]
            &\phantom{\leq}+\left\vert
            \frac{\nu^X[x\,\omega_\alpha(\cdot,\cony_\beta(\nu^Y,\cdot))]}{\mu^X[\omega_\alpha(\cdot,\cony_\beta(\mu^Y,\cdot))]}-
            \frac{\nu^X[x\,\omega_\alpha(\cdot,\cony_\beta(\nu^Y,\cdot))]}{\mu^X[\omega_\alpha(\cdot,\cony_\beta(\nu^Y,\cdot))]}
            \right\vert\label{eq:conx_est_bound_three}
            \tag{III}\\[.25cm]
            &\phantom{\leq}+\left\vert
            \frac{\nu^X[x\,\omega_\alpha(\cdot,\cony_\beta(\nu^Y,\cdot))]}{\mu^X[\omega_\alpha(\cdot,\cony_\beta(\nu^Y,\cdot))]}-
            \frac{\nu^X[x\,\omega_\alpha(\cdot,\cony_\beta(\nu^Y,\cdot))]}{\nu^X[\omega_\alpha(\cdot,\cony_\beta(\nu^Y,\cdot))]}
            \right\vert.\label{eq:conx_est_bound_four}
            \tag{IV}
        \end{align}
        \underline{\textit{Step 1: Bounding \eqref{eq:conx_est_bound_one}:}} 
        We need to devise an estimate of the form
        \begin{align}
            \label{eq:bound_conx_one}
            \eqref{eq:conx_est_bound_one} \leq C\; W_p(\mu^Y,\nu^Y)
        \end{align}
        for some constant $C>0$ depending on $p$ and $R$. Thus, we first devise the representation
        \begin{align*}
            \eqref{eq:conx_est_bound_one} = \frac{|\mu^X[x\,\omega_\alpha(\cdot,\cony_\beta(\mu^Y,\cdot))]-\mu^X[x\,\omega_\alpha(\cdot,\cony_\beta(\nu^Y,\cdot))]|}{\mu^X[\omega_\alpha(\cdot,\cony_\beta(\mu^Y,\cdot))]}.
        \end{align*}
        We can bound the denominator using \ref{eq:lower_bnd_weight} of Lemma \ref{lem:est_weighted_mean}. Then, for the numerator we compute using \ref{ass:bnd} and \ref{ass:lip}
       \begin{align*}
            &|\mu^X[x\,\omega_\alpha(\cdot,\cony_\beta(\mu^Y,\cdot))]-\mu^X[x\,\omega_\alpha(\cdot,\cony_\beta(\nu^Y,\cdot))]|\\
            &\leq \int_{\RR^{d_1}}|x\,\omega_\alpha(x,\cony_\beta(\mu^Y,x))-x\,\omega_\alpha(x,\cony_\beta(\nu^Y,x))|\diff\mu^X(x)\\
            &\leq C\int_{\RR^{d_1}}|x|\,(1+|x|+|\cony_\beta(\nu^Y,x)|+|\cony_\beta(\mu^Y,x)|)^s\\
           & \phantom{\leq C\int_{\RR^{d_1}}}\times|\cony_\beta(\mu^Y,x)-\cony_\beta(\nu^Y,x)|\diff\mu^X(x).
        \end{align*}
        We use \ref{eq:upper_bnd_cons} of Lemma \ref{lem:est_weighted_mean} and \ref{est:y_msr_bas} of Lemma \ref{lem:stab_est_y} to find a constant $C>0$, depending only on $p$ and $R$, such that
        \begin{align*}
            &\int_{\RR^{d_1}}|x|\,(1+|x|+|\cony_\beta(\nu^Y,x))|+|\cony_\beta(\mu^Y,x))|)^s\\
            &\phantom{\int_{\RR^{d_1}}}\times|\cony_\beta(\mu^Y,x)-\cony_\beta(\nu^Y,x)|\diff\mu^X(x)\\
            &\leq C\;W_p(\mu^Y,\nu^Y) \int_{\RR^{d_1}} (1+|x|)^{2s+1}\diff\mu^X(x).
        \end{align*}
        By using the Jensen inequality and the fact that $2s+1\leq p$, we find that the inequality \eqref{eq:bound_conx_one} holds true.
        
        \noindent \underline{\textit{Step 2: Bounding \eqref{eq:conx_est_bound_two}:}} 
        We need to devise an estimate of the form
        \begin{align}
            \label{eq:bound_conx_two}
            \eqref{eq:conx_est_bound_two} \leq C\; W_p(\mu^X,\nu^X)
        \end{align}
        for some constant $C>0$ depending on $p$ and $R$. Thus, we first devise the representation
        \begin{align*}
            \eqref{eq:conx_est_bound_two} = \frac{|\mu^X[x\,\omega_\alpha(\cdot,\cony_\beta(\nu^Y,\cdot))]-\nu^X[x\,\omega_\alpha(\cdot,\cony_\beta(\nu^Y,\cdot))]|}{\mu^X[\omega_\alpha(\cdot,\cony_\beta(\mu^Y,\cdot))]}.
        \end{align*}
        We can bound the denominator using \ref{eq:lower_bnd_weight} of Lemma \ref{lem:est_weighted_mean}. Then for the numerator we take any coupling $\gamma\in \Gamma(\mu^X,\nu^X)$ and compute
        \begin{align}
            \label{eq:conx_est_bound_two_bound_num}
            \begin{split}
                &|\mu^X[x\,\omega_\alpha(\cdot,\cony_\beta(\nu^Y,\cdot))]-\nu^X[x\,\omega_\alpha(\cdot,\cony_\beta(\nu^Y,\cdot))]|\\
            &\leq \int_{\RR^{d_1}\times\RR^{d_1} }|x\,\omega_\alpha(x,\cony_\beta(\nu^Y,x))-y\,\omega_\alpha(y,\cony_\beta(\nu^Y,y))|\diff\gamma(x,y)\\
            &\leq \int_{\RR^{d_1}\times\RR^{d_1} }|x-y|\,\omega_\alpha(x,\cony_\beta(\nu^Y,x))\diff\gamma(x,y)\\
            &\phantom{\leq} +\int_{\RR^{d_1}\times\RR^{d_1} }|y|\, |\omega_\alpha(x,\cony_\beta(\nu^Y,x))-\omega_\alpha(y,\cony_\beta(\nu^Y,y))|\diff\gamma(x,y).
            \end{split}
        \end{align}
        For the first term on the RHS of \eqref{eq:conx_est_bound_two_bound_num}, we use \ref{ass:bnd} and the Jensen inequality to estimate
        \begin{align}
            \label{eq:conx_est_bound_two_part1}
            \begin{split}
                &\int_{\RR^{d_1}\times\RR^{d_1} }|x-y|\,\omega_\alpha(x,\cony_\beta(\nu^Y,x))\diff\gamma(x,y)\leq e^{\alpha c}\; \left(\int_{\RR^{d_1}\times \RR^{d_1}} |x-y|^p\diff \gamma(x,y)\right)^{\frac{1}{p}}.
            \end{split}
        \end{align}
        For the second term on the RHS of \eqref{eq:conx_est_bound_two_bound_num} we use \ref{ass:lip} to bound
        \begin{align*}
            &\int_{\RR^{d_1}\times\RR^{d_1} }|y|\; |\omega_\alpha(x,\cony_\beta(\nu^Y,x))-\omega_\alpha(y,\cony_\beta(\nu^Y,y))|\diff\gamma(x,y)\\
            &\leq C\int_{\RR^{d_1}\times\RR^{d_1} }|y|\; (1+|x|+|y|+|\cony_\beta(\nu^Y,x)|+|\cony_\beta(\nu^Y,y)|)^s\\
            &\phantom{\leq C\int_{\RR^{d_1}\times\RR^{d_1} }}\times \big(|x-y|+|\cony_\beta(\nu^Y,x)-\cony_\beta(\nu^Y,y)|\diff\gamma(x,y).
        \end{align*}
        Then we use \ref{eq:upper_bnd_cons} of Lemma \ref{lem:est_weighted_mean} and \ref{est:y_var_bas} of Lemma \ref{lem:stab_est_y} to further estimate the above term by
        \begin{align*}
            &\int_{\RR^{d_1}\times\RR^{d_1} }|y|\; (1+|x|+|y|+|\cony_\beta(\nu^Y,x)|+|\cony_\beta(\nu^Y,y)|)^s\\
            &\phantom{\int_{\RR^{d_1}\times\RR^{d_1} }}\times\big(|x-y|+|\cony_\beta(\nu^Y,x)-\cony_\beta(\nu^Y,y)|\diff\gamma(x,y)\\
            &\leq 
            C\int_{\RR^{d_1}\times\RR^{d_1} } (1+|x|+|y|)^{2s+1} \; |x-y|\diff \gamma(x,y)
        \end{align*}
        for some constant $C>0$ depending only only on $p$ and $R$. Next, we use the Hölder inequality to compute
        \begin{align}
            \label{eq:conx_est_bound_two_part2}
            \begin{split}
                &\int_{\RR^{d_1}\times\RR^{d_1} } (1+|x|+|y|)^{2s+1} \, |x-y|\diff \gamma(x,y)\\
            &
            \leq \left(\int_{\RR^{d_1}\times\RR^{d_1} } (1+|x|+|y|)^{\frac{(2s+1)p}{p-1}}\diff \gamma(x,y)\right)^{\frac{p-1}{p}}\times\left(\int_{\RR^{d_1}\times\RR^{d_1} } |x-y|^p\diff \gamma(x,y)\right)^{\frac{1}{p}}.
            \end{split}
        \end{align}
        The claim now follows by using $\frac{(2s+1)p}{p-1}\leq p$ in \eqref{eq:conx_est_bound_two_part2}, combining with \eqref{eq:conx_est_bound_two_part1} and minimizing over all $\gamma\in \Gamma(\mu^X,\nu^X)$.

        \noindent\underline{\textit{Step 3: Bounding \eqref{eq:conx_est_bound_three} and \eqref{eq:conx_est_bound_four}:}} We compute
        \begin{align*}
            \eqref{eq:conx_est_bound_three}=|\nu^X[x\,\omega_\alpha(\cdot,\cony_\beta(\nu^Y,\cdot))]|\,\frac{|\mu^X[\omega_\alpha(\cdot,\cony_\beta(\mu^Y,\cdot))]-\mu^X[\omega_\alpha(\cdot,\cony_\beta(\nu^Y,\cdot))] |}{\mu^X[\omega_\alpha(\cdot,\cony_\beta(\mu^Y,\cdot))]\, \mu^X[\omega_\alpha(\cdot,\cony_\beta(\nu^Y,\cdot))]}.
        \end{align*}
        We only need to bound the difference, since we bound the denominator from below with \ref{eq:lower_bnd_weight} and the numerator from above with \ref{eq:upper_bnd_weighted_mean}, both from Lemma \ref{lem:est_weighted_mean}. But the difference we estimate in an analogue way as for \eqref{eq:conx_est_bound_one}. We estimate \eqref{eq:conx_est_bound_four} in a similar way, but we follow the same procedure as for \eqref{eq:conx_est_bound_two}. With this we conclude the proof.
    \end{proof}
    
    \begin{corollary}[Stability estimate for $\cony_\beta$]
    	\label{cor:stab_est_cony}
    	Suppose that Assumptions \ref{ass:bnd} and \ref{ass:lip} hold. Then, for all $R>0$ and $p\geq s+2$, there exists a constant $C>0$ depending on $R$ and $p$ such that the followin asserstions holds true:
        \begin{enumerate}[label=(S\arabic*)]
            \item\label{est:y_msr} For all $\mu^Y\in \rP_{p,R}(\RR^{d_2})$, $\nu^Y\in\rP_p(\RR^{d_2})$ and $U\in\RR^{d_1}$ we have
        	\begin{align*}
        		|\cony_\beta(\mu^Y, U)-\cony_\beta(\nu^Y,U)|
        		\leq C\,(1+|U|)^s\;W_p(\mu^Y,\nu^Y).
        	\end{align*}
            \item\label{est:y_var_msr} For all $\mu^Y\in \rP_{p,R}(\RR^{d_2})$, $\nu^Y\in\rP_p(\RR^{d_2})$ and $U,V\in\RR^{d_1}$ with $|U|\leq R$ we have
        	\begin{align*}
        		|\cony_\beta(\mu^Y, U)-\cony_\beta(\nu^Y,V)|
        		\leq C\,\big(W_p(\mu^Y,\nu^Y)+|U-V|\big).
        	\end{align*}
        \end{enumerate}
    \end{corollary}
    \begin{proof}
    We first prove \ref{est:y_msr}. Suppose for a contradiction that the assumption were false. Then, there exist sequences $\{(\mu_n^Y,\nu_n^Y)\}_{n\in\NN}\subseteq \rP_{p,R}(\RR^{d_2})\times \rP_p(\RR^{d_2})$ and $\{U_n\}_{n\in\NN}\subseteq \RR^{d_1}$ with
        \begin{align*}
            \frac{|\cony_\beta(\mu_n^Y,U_n)-\cony_\beta(\nu_n^Y,U_n)|}{(1+|U_n|)^s\, W_p(\mu_n^Y,\nu_n^Y)}\geq n.
        \end{align*}
        Clearly, $W_p(\nu_n^Y,\delta_0)\to \infty$, otherwise the above claim would be false by \ref{est:y_msr_bas} from Lemma \ref{lem:stab_est_y}. Then note by the (reverse) triangle inequality and \ref{eq:upper_bnd_cons} from Lemma \ref{lem:est_weighted_mean} we derive the contradiction
        \begin{align*}
            n\leq C\,\frac{W_p(\nu_n^Y,\delta_0)+R}{(1+|U_n|)^s\, \big(W_p(\nu_n^Y,\delta_0)-R\big)},
        \end{align*}
        where $C>0$ is some constant depending on $p$. Thus the claim must hold by proof of contradiction. Next we prove \ref{est:y_var_msr}. Let us define the probability measure $\bar{\mu}^Y:=\mu^Y\otimes\delta_U$ and $\bar{\nu}^Y:=\nu^Y\otimes\delta_V$. We compute the $p$-mean of $\bar{\mu}^Y$ by using the Jensen inequality
    \begin{align*}
        \int_{\RR^{d_2}\times \RR^{d_1}} |z|^p \diff \bar{\mu}^Y(z)
        \leq 2^{p-1}\left(\int_{\RR^{d_2}} |y|^p \diff \mu^Y(y)+|U|^p\right)
        \leq 2^p R. 
    \end{align*}
    Then from \cite[Corollary 3.3]{gerber_hoffmann_vaes_2025_meanfield} we know there exists a constant $C>0$ depending only on $R$ and $p$ such that
    	\begin{align*}
    		&|\cony_\beta(\mu^Y,U)-\cony_\beta(\nu^Y,V)|\leq C W_p(\bar \mu^Y,\bar \nu^Y)
    		\leq C 2^{\frac{p-1}{p}}\big( W_p(\mu^Y,\nu^Y)+|U-V|\big).
    	\end{align*}
    With this we conclude the proof.
    \end{proof}

    \begin{corollary}[Stability Estimate for $\conx_{\alpha,\beta}$]
        \label{cor:stab_est_conx}
        Let $R>0$ and $p\geq 2(s+1)$. Then there exists a constant $C>0$ depending on $R$ and $p$, such that for all $\mu^X\in\rP_{p,R}(\RR^{d_1})$, $\mu^Y\in\rP_{p,R}(\RR^{d_2})$, $\nu^X\in \rP_p(\RR^{d_1} )$ and $\nu^Y\in \rP_p(\RR^{d_2})$ we have
        \begin{align*}
            |\conx_{\alpha,\beta}(\mu^X,\mu^Y)-\conx_{\alpha,\beta}(\nu^X,\nu^Y)|
            \leq C (W_p(\mu^X,\nu^X)+W_p(\mu^Y,\nu^Y)).
        \end{align*}
    \end{corollary}
    \begin{proof}
        Suppose that the claim does not hold true, then there exist sequences
        \begin{align*}
        \{(\mu_n^X,\mu_n^Y)\}_{n\in\NN}&\subseteq \rP_{p,R}(\RR^{d_1})\times \rP_{p,R}(\RR^{d_2}),\\
        \{(\nu_n^X,\nu_n^Y)\}_{n\in\NN}&\subseteq \rP_p(\RR^{d_1})\times \rP_p(\RR^{d_2})
        \end{align*}
        that satisfy
        \begin{align*}
            \frac{|\conx_{\alpha,\beta}(\mu_n^X,\mu_n^Y)-\conx_{\alpha,\beta}(\nu_n^X,\nu_n^Y)|}{W_p(\mu_n^X,\nu_n^X)+W_p(\mu_n^Y,\nu_n^Y)}\geq n.
        \end{align*}
        Clearly, $W_p(\nu_n^X,\delta_0)\wedge W_p(\nu_n^Y,\delta_0)\to \infty$, otherwise the above claim would be false by Lemma \ref{lemma:basic_stability_estimate}. Then note by the (reverse) triangle inequality and \ref{eq:upper_bnd_cons} of Lemma \ref{lem:est_weighted_mean} we derive the contradiction
        \begin{align*}
            n\leq C\cdot\frac{W_p(\nu_n^X,\delta_0)+R}{W_p(\nu_n^X,\delta_0)+ W_p(\nu_n^Y,\delta_0)-2R}.
        \end{align*}
        Thus the claim must hold by proof of contradiction.
    \end{proof}

    \subsection{Moment estimates for \eqref{eq:cbo} and \eqref{eq:mf_cbo} dynamics}
    \label{section:moment_estimate_for_cbo_dynamics}
    In this section, we control the moments of the particles in \eqref{eq:cbo} and \eqref{eq:mf_cbo}. We begin with the estimation of \eqref{eq:cbo}.
     
    \begin{lemma}[Moment estimates for the empirical measures]
    	\label{lemma:moment_estimates_for_the_empirical_measures}
    	Let Assumptions \ref{ass:bnd} and \ref{ass:lip} hold and let $p\geq 2$, $\bar\rho_0^X\in \rP_p(\RR^{d_1})$, $\bar\rho_0^Y\in \rP_p(\RR^{d_2})$. Let $\{(X^{i},Y^{i})\}_{i\in [N]}$ be solutions to the SDEs \eqref{eq:cbo} with the i.i.d. initial data satisfying $\law(\bX_0^i,\bY_0^i)= \dmf_0^X\otimes \dmf_0^Y$. Then there exists a constant $\kappa>0$ that does not depend on $N$ such that
    	\begin{align*}
    		\sup_{\substack{ i \in [N] }}\EE \left[ \sup_{t \in [0, T]} 
    		\;|X_t^{i}|^p +|Y_t^{i}|^p 
    		\right]\leq \kappa\, \big(\EE|X_0^{1}|^p+\EE|Y_0^{1}|^p\big).
    	\end{align*}
    \end{lemma}
    \begin{proof}
    First, there exists a constant $C>0$ independent of $N$ such that we can estimate the drift and diffusion terms by
    	\begin{align}
            \label{eq:bnd_drift_1}
    		\begin{split}
    		    &|X_t^{i}-\conx_{\alpha,\beta}(\dcbo_t^{X,N},\dcbo_t^{Y,N})|^p
    		\vee \|D(X_t^{i}-\conx_{\alpha,\beta}(\dcbo_t^{X,N},\dcbo_t^{Y,N}))\|^p\leq C\left(|X_t^{i}|^p+\int_{\RR^d}|x|^p\diff\dcbo_t^{X,N}(x)\right)^{\frac{1}{p}}
    		\end{split}
    	\end{align}
        and 
        \begin{align}
            \label{eq:bnd_drift_2}
    		\begin{split}
    		    &|Y_t^{i}-\cony_\beta(\rho_t^{Y,N},X_t^i)|^p
    		\vee \|D(Y_t^{i}-\cony_\beta(\rho_t^{Y,N},X_t^i))\|^p\leq C\left(|Y_t^{i}|^p+\int_{\RR^d}|y|^p\diff\dcbo_t^{Y,N}(y)\right)^{\frac{1}{p}},
    		\end{split}
    	\end{align}
        where $\|\cdot\|$ denotes the Frobenius norm, and we have used the consensus estimate \ref{eq:upper_bnd_cons} of Lemma \ref{lem:est_weighted_mean}. Next we apply the Burkholder-Davis-Grundy (BDG) inequality \cite[Chapter 1, Theorem 7.3]{mao_2011_stochastic} and the bounds \eqref{eq:bnd_drift_1} and \eqref{eq:bnd_drift_2} to find for all $t\in [0,T]$ that
    	\begin{align*}
    		&\frac{1}{3^{p-1}}\EE\left[\sup_{s\in [0,t]}\left\vert X_s^{i} \right\vert^p\right]\\
    		&\leq \EE|X_0^{i}|^p+\lambda^p T^{p-1}\int_0^t \EE\left\vert X_s^{i}-\conx_{\alpha,\beta}(\dcbo_s^{X,N},\dcbo_s^{Y,N}) \right\vert^p\diff s\\
    		&\phantom{\leq}+T^{\frac{p}{2}-1}\sigma^p C_{\mathrm{DBG}} \int_0^t \EE\left\vert D(X_s^{i}-\conx_{\alpha,\beta}(\dcbo_s^{X,N},\dcbo_s^{Y,N})) \right\vert^p\diff s\\
    		&\leq \EE|X_0^{i}|^p+C\left(\int_0^t \EE\left[|X_s^{i}|^p+\int_{\RR^d}|x|^p\diff\dcbo_s^{X,N}(x)\right]\diff s\right).
    	\end{align*}
    	As the particles $X^{1},\ldots,X^{N}$ are exchangeable, we see that
    	\begin{align*}
    		\EE\left[\int_{\RR^d}|x|^p\diff\dcbo_t^{X,N}(x)\right]=\EE|X_t^{i}|^p.
    	\end{align*}
    	Hence, for some constant $C>0$, we derive
    	\begin{align*}
    		\EE\left[\sup_{s\in [0,t]}\left\vert X_s^{i} \right\vert^p\right]
    		\leq C\left(\EE|X_0^{i}|^p+\int_0^t \EE\left[\sup_{r\in [0,s]}\left\vert X_r^{i} \right\vert^p\right]\diff s\right).
    	\end{align*}
    	Similarly we also have
        \begin{align*}
    		\EE\left[\sup_{s\in [0,t]}\left\vert Y_s^{i} \right\vert^p\right]
    		\leq C\left(\EE|Y_0^{i}|^p+\int_0^t \EE\left[\sup_{r\in [0,s]}\left\vert Y_r^{i} \right\vert^p\right]\diff s\right).
    	\end{align*}
    	The result now follows from the Grönwall inequality.
    \end{proof}

    In the exact same way one can also obtain the moment bounds for the mean-field dynamics \eqref{eq:mf_cbo}.
     \begin{lemma}\label{lem:moment_estimate_mf_cbo}
         Let $p\geq 2$ suppose that the continuous process $(\bX,\bY)$ satisfy the stochastic equation \eqref{eq:mf_cbo} on $[0,T]$, and the initial data satisfies $\EE|\bX_0|^p+\EE|\bY_0|^p<\infty$. Then there exists a constant $\kappa>0$ that does not depend on $N$ such that
        \begin{equation}
            \label{eq:moment_bnd_mf_cbo}
            \EE\left[\sup_{t\in[0,T]}|\bX_{t }|^p+|\bY_{t}|^p\right]\leq \kappa\,\big(\EE|\bX_0|^p+\EE|\bY_0|^p\big).
        \end{equation}
    \end{lemma}
  \begin{remark}
        \label{remark:add_xi_to_cbo}
        If we exchange the SDE \eqref{eq:mf_cbo} with the one below
          \begin{align*}
        \begin{cases}
            \diff\bX_t = -\lambda(\bX_t - \xi\conx_{\alpha,\beta}(\dmf_t^X,\dmf_t^Y))\diff t + \sigma D(\bX_t -\xi \conx_{\alpha,\beta}(\dmf_t^X,\dmf_t^Y))\diff B_t^X, \\
            \diff \bY_t = -\lambda(\bY_t - \xi\cony_{\beta}(\dmf_t^Y, \bX_t))\diff t + \sigma D(\bY_t - \xi\cony_{\beta}(\dmf_t^Y, \bX_t))\diff B_t^Y,
        \end{cases}
    \end{align*}
        for any $\xi\in [0,1]$, then the Lemma \ref{lem:moment_estimate_mf_cbo} still holds. The proof of this is analogue. This fact will be used in the proof of Theorem \ref{theorem:existence_and_uniqueness_for_mfcbo}.
    \end{remark}

    \subsection{Convergence of the weighted mean for i.i.d. samples}
    \label{section:meanfield_limit_of_consensus_for_mfcbo}
    In this section we utilize the lemma \ref{lem:unif_emp_msr_conv} to prove a uniform in variable and law-of-large-number estimate for the consensus functions $\conx_{\alpha,\beta}$ and $\cony_\beta$. 

    \begin{lemma}[Convergence of the weighted mean for i.i.d. samples in $Y$]
    	\label{lemma:convergence_of_the_weighted_mean_for_iid_samples_y}
         Let Assumptions \ref{ass:bnd} and \ref{ass:lip} hold, let $2\leq p<r$, $R>0$. Then for all $\mu^Y\in\rP_r(\RR^{d_2})$ there exists constant $C>0$ depending only on $\mu$, $R$, $p$ and $r$ such that
         \begin{align}
            \label{eq:quant_conv_cony}
             \sup_{x\in\RR^{d_1}}\EE|\cony_\beta(\mu^{Y,N},x)-\cony_\beta(\mu^Y,x)|^p
    	 	\leq C\, N^{-\frac{p}{2}},
         \end{align}
         where we define the empirical measure and random variables by
        \begin{gather*}
    	 	\mu^{Y,N}=\frac{1}{N}\sum_{j=1}^N \delta_{Y^j},\qquad
    	 	\{Y^j\}_{j\in\NN}\overset{\text{i.i.d.}}{\sim} \mu^Y.
    	 \end{gather*}
    \end{lemma}
    \begin{proof}
        We will use Lemma \ref{lem:unif_emp_msr_conv} to prove the convergence of the weighted mean for i.i.d. samples. Set $q:=\frac{p(r+2)}{r-p}$ and $m:=\frac{pq}{q-p}$. We set 
        \begin{align*}
            \omega_{x,j}:=\omega_{-\beta}^\rE(x, V_j),\quad V_j:=Y^j.
        \end{align*} 
        It suffices to prove the uniform limit \eqref{eq:unif_emp_msr_conv_cond} from Lemma \ref{lem:unif_emp_msr_conv}, which in this context is
        \begin{align}
        \label{eq:unif_est_to_verify}
        \sup_{\substack{x\in\RR^{d_1} \\ j\in \NN}}\;\EE[\omega_{x,j}^q]+\EE|\omega_{x,j}V_j|^m+N_x+\frac{1}{D_x}
        <\infty,
    \end{align}
    where $N_x:=\EE|\omega_{x,1}V_1|$ and $D_x:=\EE[\omega_{x,1}]$. Then, we bound the first time of \eqref{eq:unif_est_to_verify} by using \ref{ass:bnd}
    \begin{align*}
        \EE[\omega_{x,j}^q]
        =\int_{\RR^{d_2}} |\omega_{-\beta}^{\rE}(x,y)|^q\diff\mu^Y(y)
        \leq e^{\beta c q},
    \end{align*}
    and we bound the second term of \eqref{eq:unif_est_to_verify} by using \ref{ass:bnd}, the fact that $m\leq r$, and the Jensen inequality by
    \begin{align*}
        \EE|\omega_{x,j}V_j|^m
        \leq e^{\beta c m} \left(\int_{\RR^{d_2}} |y|^r\diff \mu^Y(y)\right)^{\frac{m}{r}}.
    \end{align*}
    We can bound the third term of \eqref{eq:unif_est_to_verify} by the second one via the Jensen inequality, and the last term of \eqref{eq:unif_est_to_verify} follows by \ref{eq:lower_bnd_weight} of Lemma \ref{lem:est_weighted_mean}. Thus, the convergence \eqref{eq:quant_conv_cony} is a direct application of Lemma \ref{lem:unif_emp_msr_conv}.
    \end{proof}
    
    \begin{lemma}[Convergence of the weighted mean for i.i.d. samples]
    	\label{lemma:convergence_of_the_weighted_mean_for_iid_samples}
         Let Assumptions \ref{ass:bnd} and \ref{ass:lip} hold and let $2\leq p<r$, $R>0$, $s+1\leq r$ and $\nu_0\in\rP_r(\RR^{d_1})$. Then for all $\mu^X\in\rP_r(\RR^{d_1})$ and $\mu^Y\in \rP_r(\RR^{d_2})$ there exists constant $C>0$ depending only on $\mu^X$, $R$, $p$ and $r$ such that for all $N\in\NN$ we have
             \begin{align}
                 \label{lem:conv_wmeans_for_iid_samp:outer}
                 \sup_{\nu\in \rP_r(\RR^{d_2})}\EE|\conx_{\alpha,\beta}(\mu^{Y,N},\nu)-\conx_{\alpha,\beta}(\mu^X,\nu)|^p
    	 	&\leq C N^{-\frac{p}{2}},\\
            \label{lem:conv_wmeans_for_iid_samp:inner}
            \EE|\conx_{\alpha,\beta}(\nu_0,\mu^{Y,N})-\conx_{\alpha,\beta}(\nu_0,\mu^Y)|^p
    	 	&\leq C N^{-\frac{p}{2}},
             \end{align}
        where we define the empirical measures and random variables by
        \begin{gather*}
    	 	\mu^{Z,N}=\frac{1}{N}\sum_{j=1}^N \delta_{Z^j},\qquad
    	 	\{Z^j\}_{j\in\NN}\overset{\text{i.i.d.}}{\sim} \mu^Z, \qquad \text{where }Z\in\{X,Y\}.
    	 \end{gather*}
    \end{lemma}
    \begin{proof}
        The proof of \eqref{lem:conv_wmeans_for_iid_samp:outer} is exactly analogue to the proof of Lemma \ref{lemma:convergence_of_the_weighted_mean_for_iid_samples_y}. Thus we only address \eqref{lem:conv_wmeans_for_iid_samp:inner}. We use both \ref{eq:lower_bnd_weight} and \ref{eq:upper_bnd_weighted_mean} of Lemma \ref{lem:est_weighted_mean} to find a constant $C>0$, independent of $N$, such that the following equation holds (as was done in \eqref{eq:stab_est_cons_diff_meas_start} of Lemma \ref{lem:stab_est_y})
        \begin{align*}
            &\EE|\conx_{\alpha,\beta}(\nu_0,\mu^{Y,N})-\conx_{\alpha,\beta}(\nu_0,\mu^Y)|^p\\
            &=\left\vert\frac{\displaystyle\int_{\RR^{d_1}}x\,\omega_\alpha(x,\cony_\beta(\mu^{Y,N},x))\diff \nu_0(x)}{\displaystyle\int_{\RR^{d_1}}\omega_\alpha(x,\cony_\beta(\mu^{Y,N},x))\diff \nu_0(x)}
            -\frac{\displaystyle\int_{\RR^{d_1}}x\,\omega_\alpha(x,\cony_\beta(\mu^Y,x))\diff \nu_0(x)}{\displaystyle\int_{\RR^{d_1}}\omega_\alpha(x,\cony_\beta(\mu^Y,x))\diff \nu_0(x)}\right\vert\\
            &\leq C\, \left\vert\int_{\RR^{d_1}}x\,\omega_\alpha(x,\cony_\beta(\mu^{Y,N},x))-x\,\omega_\alpha(x,\cony_\beta(\mu^Y,x))\diff \nu_0(x)\right\vert\\
            &\phantom{\leq}+C\,\left\vert\int_{\RR^{d_1}}\omega_\alpha(x,\cony_\beta(\mu^{Y,N},x))-\omega_\alpha(x,\cony_\beta(\mu^Y,x))\diff \nu_0(x)\right\vert.
        \end{align*}
        We now use \ref{ass:bnd} and \ref{ass:lip} to find a constant $C>0$ depending on $\alpha$, $c$ and $\nu_0$ such that
        \begin{align*}
            & \left\vert\int_{\RR^{d_1}}x\,\omega_\alpha(x,\cony_\beta(\mu^{Y,N},x))-x\,\omega_\alpha(x,\cony_\beta(\mu^Y,x))\diff \nu_0(x)\right\vert\\
            &\phantom{\leq} +\left\vert\int_{\RR^{d_1}}\omega_\alpha(x,\cony_\beta(\mu^{Y,N},x))-\omega_\alpha(x,\cony_\beta(\mu^Y,x))\diff \nu_0(x)\right\vert\\
            &\leq C\int_{\RR^{d_1}}(1+|x|)^{s+1}\,|\cony_\beta(\mu^{Y,N},x)-\cony_\beta(\mu^Y,x)|\diff \nu_0(x).
        \end{align*}
        Next we employ the convergence \eqref{eq:quant_conv_cony} in Lemma \ref{lemma:convergence_of_the_weighted_mean_for_iid_samples_y}, the Jensen inequality and the fact that $s+1\leq r$ to get
        \begin{align*}
            &\int_{\RR^{d_1}}(1+|x|)^{s+1}\; |\cony_\beta(\mu^{Y,N},x)-\cony_\beta(\mu^Y,x)|\diff \nu_0(x)\leq C \left(\int_{\RR^{d_1}}(1+|x|)^r\diff \nu_0(x)\right)^{\frac{s+1}{r}}
            \,N^{-\frac{p}{2}}
            \leq C\,N^{-\frac{p}{2}},
        \end{align*}
        where $C>0$ is a constant that changes from line to line. With this we conclude the proof
    \end{proof}

\section{Proof of Theorem \ref{theorem:existence_and_uniqueness_for_cbo}}
    \label{section:wellposedness_for_cbo}

   In this section we prove the well-posedness of CBO particle system \eqref{eq:cbo} by using the non-explosion criterion from stochastic Lyapunov theory as stated in \cite[Theorem 3.5]{khasminskii_2012_stochastic}. To that end, we rewrite \eqref{eq:cbo} as
    \begin{align*}
        \diff (\bfX_t,\bfY_t) =F(\bfX_t,\bfY_t)\diff t+G(\bfX_t,\bfY_t)\diff \mathbf{B}_t,
    \end{align*}
    where we use the notation
    \begin{align*}
        (\bfX_t,\bfY_t)&:=(X_t^{1},\ldots, X_t^{N},Y_t^{1},\ldots,Y_t^{N}), \\
        \mathbf{B}_t&:=(B_t^{X,1},\ldots, B_t^{X,N},B_t^{Y,1},\ldots,B_t^{Y,N})
    \end{align*}
    and we define 
    \begin{align*}
        F(\bfX_t,\bfY_t)&:=-\lambda \begin{pmatrix}
            X_t^{1}-\conx_{\alpha,\beta}(\dcbo_t^{X,N}, \dcbo_t^{Y,N}) \\ \vdots \\ X_t^{N}-\conx_{\alpha,\beta}(\dcbo_t^{X,N}, \dcbo_t^{Y,N}) \\
            Y_t^{1}-\cony_\beta(\dcbo_t^{Y,N},X_t^1) \\ \vdots \\
            Y_t^{N}-\cony_\beta(\dcbo_t^{Y,N},X_t^N)
        \end{pmatrix},\\ 
        G(\bfX_t,\bfY_t)&:=\sigma\begin{pmatrix}
            D(X_t^{1}-\conx_{\alpha,\beta}(\dcbo_t^{X,N}, \dcbo_t^{Y,N}))\\
            & \ddots  \\
            && D(Y_t^{N}-\cony_\beta(\dcbo_t^{Y,N},X_t^N))
        \end{pmatrix}.
    \end{align*}
    for the function $F:\RR^{N\cdot d_1+N\cdot d_2}\to \RR^{N\cdot d_1+N\cdot d_2}$ and the matrix valued function $G:\RR^{N\cdot d_1+N\cdot d_2}\to\RR^{(N\cdot d_1+N\cdot d_2)\times (N\cdot d_1+N\cdot d_2)}$. We have sublinear growth in the consensus
    \begin{align}
        \label{eq:sub_lin_cons}
        \begin{split}
        |\conx_{\alpha,\beta}(\dcbo_t^{X,N}, \dcbo_t^{Y,N})|
        &\leq \sum_{i=1}^N \frac{e^{-\alpha \rE(X_t^{i}, \cony_{\beta}(\dcbo_t^{Y,N},X_t^i))}}{\sum_{j=1}^N e^{-\alpha \rE(X_t^{j}, \cony_{\beta}(\dcbo_t^{Y,N},X_t^i))}}\, |X_t^{i}|=\sum_{i=1}^N |X_t^{i}|
        \leq \sqrt{N}|\bfX_t|,\\
    \end{split}
    \end{align}
    and
    \begin{align}
        \label{eq:sub_lin_cons_2}
        \begin{split}
        |\cony_{\beta}(\dcbo_t^{X,N},X_t^n)|
        &\leq \sum_{i=1}^N \frac{e^{\beta \rE(Y_t^{i}, X
        _t^n)}}{\sum_{j=1}^N e^{\beta \rE(Y_t^{j},X_t^n)}}\, |Y_t^{i}| =\sum_{i=1}^N |Y_t^{i}|
        \leq \sqrt{N}|\bfY_t|.
        \end{split}
    \end{align}
Moreover, the above functions are locally Lipschitz.
    \begin{lemma}[Local Lipschitz property]
        \label{cor:loc_lip_prop}
        Let Assumption \ref{ass:lip} hold. We define $\bfX:=(X^{1},\ldots,X^{N})$. Then for all $R>0$, there exists a constant $C>0$ depending on $R$ and $N$ such that, for all $(\bfX,\bfY),(\bfX',\bfY')\in \RR^{N\cdot(d_1+d_2)}$ with $|(\bfX,\bfY)|\wedge |(\bfX',\bfY')|\leq R$,
        \begin{align}
            \label{eq:loc_lip_cons}
               \notag&|\conx_{\alpha,\beta}(\dcbo^{X,N}, \dcbo^{Y,N})-\conx_{\alpha,\beta}(\dcbo^{X',N},\dcbo^{Y',N})|
            +|\cony_\beta(\dcbo^{Y,N},X^i)-\cony_\beta(\dcbo^{Y',N},X^{i'})|\\
            &\leq C |(\bfX-\bfX',\bfY-\bfY')|,
        \end{align}
       where $\dcbo^{Z,N}:= \frac{1}{N}\sum_{i=1}^N \delta_{Z^{i}}$ for Z being $X,Y,X'$ and $Y'$.
    \end{lemma}
    \begin{proof}
         We want to apply Corollaries \ref{cor:stab_est_cony} and \ref{cor:stab_est_conx}, however Assumption \ref{ass:bnd} may not necessarily hold. But note as $\mathrm{supp}(\rho^{Z,N})\subseteq \overline{B_R}$ for $Z$ being $X$, $Y$, $X'$ and $Y'$, we only evaluate the cost function $\rE$ on $\overline{B_R}\times \overline{B_R}$. This means, we could instead use any function $\tilde{\rE}$ that satisfy both Assumptions \ref{ass:bnd} and \ref{ass:lip} with $\tilde{\rE}=\rE$ on $\overline{B_R}\times \overline{B_R}$ without affecting our analysis. Thus, we substitute the function $\rE$ with $\tilde{\rE}$ and apply Corollaries \ref{cor:stab_est_cony} and \ref{cor:stab_est_conx}. So, for any $p\geq 2(s+1)$ fixed, there exists a constant $C>0$ depending on $R$ and $p$ such that
         \begin{align*}
    		&|\conx_{\alpha,\beta}(\dcbo^{X,N}, \dcbo^{Y,N})-\conx_{\alpha,\beta}(\dcbo^{X',N},\dcbo^{Y',N})|
            +|\cony_\beta(\dcbo^{Y,N},X^i)-\cony_\beta(\dcbo^{Y',N},X^{i'})|\\
    		&\leq C  \big(W_p(\dcbo^{NX},\dcbo^{X',N})+W_p(\dcbo^{Y,N},\dcbo^{Y',N})+|X^i-X^{i'}|\big)\\
            &\leq C \left(\sum_{i=1}^N |X^{i}-X^{i'}|^p\right)^{\frac{1}{p}}+C\left(\sum_{i=1}^N |Y^{i}-Y^{i'}|^p\right)^{\frac{1}{p}}+C|X^i-X^{i'}|\\
            &\leq C\sum_{i=1}^N (|X^{i}-X^{i'}|+|Y^{i}-Y^{i'}|)
            \leq C |(\bfX-\bfX',\bfY-\bfY')|.
    	\end{align*}
        Hence, we have proven the local Lipschitz property as required.
    \end{proof}
    \noindent Since $D$ is Lipschitz continuous, the conditions in \cite[Equation 3.32]{khasminskii_2012_stochastic} are satisfied. It remains to show that the coercive function $\vp(x,y):=|(x,y)|^2$ satisfies the condition \cite[Equation 3.43]{khasminskii_2012_stochastic}, this means showing there exists a constant $C>0$ with
    \begin{align*}
        &F(\bfX_t,\bfY_t)\cdot \nabla \vp(\bfX_t,\bfY_t)+\frac{1}{2} \langle G(\bfX_t,\bfY_t),\; \nabla^2\vp(\bfX_t,\bfY_t)G(\bfX_t,\bfY_t)\rangle\leq C \vp(\bfX_t,\bfY_t),
    \end{align*}
    where the above inner product denotes the dot product between matrices. The condition immediately follows from the fact that
    \begin{align*}
        &F(\bfX_t,\bfY_t)\cdot \nabla \vp(\bfX_t,\bfY_t)+\frac{1}{2} \langle G(\bfX_t,\bfY_t),\; \nabla^2\vp(\bfX_t,\bfY_t)G(\bfX_t,\bfY_t)\rangle\\
        &\leq |F(\bfX_t,\bfY_t)|^2+|(\bfX_t,\bfY_t)|^2+ \Vert G(\bfX_t,\bfY_t)\Vert^2
        \leq C\vp(\bfX_t,\bfY_t),
    \end{align*}
    where $\|\cdot\|$ denotes the Frobenius norm, and we used the sublinear growth condition \eqref{eq:sub_lin_cons}. Thus the proof of Theorem \ref{theorem:existence_and_uniqueness_for_cbo} follows from applying \cite[Theorem 3.5]{khasminskii_2012_stochastic}.

\section{Proof of Theorem \ref{theorem:existence_and_uniqueness_for_mfcbo}}
    \label{section:wellposedness_for_mfcbo}

The principal difficulty in proving the well-posedness of \eqref{eq:mf_cbo} lies in the fact that the mapping $t \mapsto \cony_\beta(\overline\rho_t^Y, \overline X_t)$ defines a stochastic process that depends on $\bX_t$. Moreover, according to estimate \ref{est:y_msr} of Corollary \ref{cor:stab_est_cony}, the function $\cony_\beta$ is only locally Lipschitz continuous. As a result, establishing compactness requires more refined tools from functional analysis. To address these challenges, we first introduce a truncated version of the equation~\eqref{eq:mf_cbo}, incorporating a truncation parameter $R > 0$ as defined in~\eqref{eq:mf_cbo1}. For the truncated model, the Leray--Schauder fixed-point theorem can be employed to establish well-posedness. The well-posedness of the original system~\eqref{eq:mf_cbo} is then obtained by letting $R \to \infty$. In the following, we divide the proof into three steps. In step 1, we establish the well-posedness of the truncated CBO model \eqref{eq:mf_cbo1} and show some properties that hold uniformly for all $R>0$. In step 2, we pass to the limit as the truncation parameter $R \to \infty$ and obtain the existence of a solution to~\eqref{eq:mf_cbo}. Finally, in step 3, we prove the uniqueness of the solution. For this proof, we define for any $R>0$ the projection mapping
\begin{align}
    \label{eq:proj}
   P_R(x):=\begin{cases}
        x,\quad &\mbox{ if } |x|\leq R,\\
        R\frac{x}{|x|},\quad &\mbox{ if } |x|>R.
    \end{cases} 
\end{align}    
    \noindent\underline{\textit{Step 1: Well-posedness of truncated model and convergence of consensus as $R\to \infty$.}}
    We introduce a truncated version of the \eqref{eq:mf_cbo} equation in the lemma below.

    \begin{lemma}\label{lem:existence_trunc_sol}
    Let $p>2$ if $s=0$, otherwise let $p\geq 2+s$. Then, for all $T,\,R > 0$, $\dmf_0^X\in\rP_p(\RR^{d_1})$ and $\dmf_0^Y\in\rP_{p+\eta}(\RR^{d_2})$, there exists unique $\bX^R:\Omega\to C^0([0,T], \RR^{d_1})$ and $\bY^R:\Omega\to C^0([0,T], \RR^{d_2})$ satisfying
    \begin{align}
        \label{eq:mf_cbo1}
        \begin{cases}
            \begin{aligned}
                \diff\bX_t^R =& -\lambda(\bX_t^R- \conx_{\alpha,\beta}(\dmf_t^{X,R}, \dmf_t^{Y,R}))\diff t+ \sigma D(\bX_t^R - \conx_{\alpha,\beta}(\dmf_t^{X,R}, \dmf_t^{Y,R}))\diff B_t^X,
            \end{aligned}\\
            \begin{aligned}
                \diff \bY_t^R =& -\lambda(\bY_t^R - \cony_{\beta}(\dmf_t^{Y,R}, P_R(\bX_t^R)))\diff t+ \sigma D(\bY_t^R - \cony_{\beta}(\dmf_t^{Y,R}, P_R(\bX_t^R)))\diff B_t^Y,
            \end{aligned}
        \end{cases}
        \tag{$\text{MF CBO}_R$}
    \end{align}
    in the strong sense with initial condition $\law(\bX_0^R,\bY_0^R)= \dmf_0^X\otimes \dmf_0^Y$. Here $P_R$ is the projection mapping defined in \eqref{eq:proj}. Furthermore, the solutions $\big\{(\bX^R,\bY^R)\big\}_{R>0}$ satisfy the following properties
    \begin{enumerate}[label=(P\arabic*)]
       \item \label{prop:uniform_p_mom_R} We have the uniform bounded $p$-moment
       \begin{align*}
           \sup_{R>0}\;\EE\left[\sup_{t\in [0,T]} |\bX_t^R|^p+|\bY_t^R|^p\right]<\infty.
       \end{align*}
       \item \label{prop:equic_con_x} The set $\{\conx_{\alpha,\beta}(\bX^R,\bY^R)_{R>0}$ is uniformly bounded and equicontinuous in $C^0([0,T],\, \RR^{d_1})$.
       \item \label{prop:equic_con_y} For any fixed $M>0$, the set $\{\cony_{\beta}(\dmf^{Y,R}, \cdot)\}_{R>0}$ is uniformly bounded and equicontinuous in $C^0([0,T]\times B_M,\, \RR^{d_2})$.
   \end{enumerate}
   \end{lemma}
    The proof of this lemma follows a classical proof strategy. Thus, we reserve the proof of this lemma for the appendix \ref{section:proof_trunc_sde}. We let $\{(\bX^R,\bY^R)\}_{R>0}$ be the solutions to the truncated SDE \eqref{eq:mf_cbo1} that we obtain from Lemma \ref{lem:existence_trunc_sol}. Then, we show that that there exists some functions $\mathcal{X}\in C^0([0,T];\,\RR^{d_1})$ and $\mathcal{Y}\in C^0([0,T]\times\RR^{d_1};\,\RR^{d_2})$ such that we have the uniform convergence $\conx_{\alpha,\beta}(\dmf^{X,R},\dmf^{Y,R})\to \mathcal{X}$ in $C^0([0,T];\,\RR^{d_1})$ and the \textit{locally} uniform convergence $\cony_\beta(\dmf^{Y,R},\cdot)\to \mathcal{Y}$ in $C^0([0,T]\times \RR^{d_1};\,\RR^{d_2})$  as $R\to \infty$.  Similar arguments have been used in \cite{fornasier2025pde} for multiple-minimizer CBO model.
    
    Let us focus our attention on the subsequences $\{\conx_{\alpha,\beta}(\dmf^{X,n},\dmf^{Y,n})\}_n$ and $\{\cony_\beta(\dmf^{Y,n},\cdot)\}_n$. By \ref{prop:equic_con_x} of Lemma \ref{lem:existence_trunc_sol} we know that $\{\conx_{\alpha,\beta}(\dmf^{X,n},\dmf^{Y,n})\}_n$ uniformly bounded and equicontinuous in $C^0([0,T];\,\RR^{d_1})$. Applying the Arzel\`{a}-Ascoli  theorem, there exits a subsequence  $\conx_{\alpha,\beta}(\dmf^{X,n},\dmf^{Y,n})$ (without relabeling) and $\mathcal{X}\in C^0([0,T];\,\RR^{d_1})$  such that 
    \begin{equation}\label{convergenceX}
        \conx_{\alpha,\beta}(\dmf^{X,n},\dmf^{Y,n})\to \mathcal{X} \quad \mbox{ uniformly in }C^0([0,T];\,\RR^{d_1})\,.
    \end{equation} 
    On the other hand, according to \ref{prop:equic_con_y} of Lemma \ref{lem:existence_trunc_sol}, we have $\{\cony_\beta(\dmf^{Y,n},\cdot)\}$ is uniformly bounded and equicontinuous in $C^0([0,T]\times B_M;\,\RR^{d_2})$ for any $M>0$. Combining the Arzel\`{a}-Ascoli theorem and diagonal selection argument, we know that there exits a subsequence $\{\cony_\beta(\dmf^{Y,n},\cdot)\}_n$ and some $\mathcal{Y}\in C^0([0,T]\times\RR^{d_1};\,\RR^{d_2})$ such that 
    \begin{equation}\label{convergenceY}
        \cony_\beta(\dmf^{Y,n},\cdot)\to \mathcal{Y}\quad \mbox{ uniformly in } C^0([0,T]\times B_M;\,\RR^{d_2}) \mbox{ for any }M>0\,.
    \end{equation}
    Moreover, $\mathcal{Y}_t$ inherits the local Lipschitz property of  $\cony_\beta(\dmf_t^{Y,n},\cdot)$. This means
    \begin{equation}\label{LipY}
       |\mathcal{Y}_t(x_1)-\mathcal{Y}_t(x_2)|\leq C(M)|x_1-x_2|\quad \mbox{ for any }x_1\in B_M\text{ and }x_2\in \RR^{d_1}\,.
    \end{equation}

\noindent\underline{\textit{Step 2: Existence of solution to \eqref{eq:mf_cbo}.}} We define the following SDE using the limit functions $\mathcal{X}_t$ and $\mathcal{Y}_t$:
    \begin{align}
        \label{eq:limit_mf_cbo}
        \begin{cases}
            \diff\bX_t^\infty = -\lambda(\bX_t^\infty- \mathcal{X}_t)\diff t + \sigma D(\bX_t^\infty - \mathcal{X}_t)\diff B_t^X,\\
            \diff \bY_t^\infty = -\lambda(\bY_t^\infty - \mathcal{Y}_t(\bX_t^\infty))\diff t + \sigma D(\bY_t^\infty - \mathcal{Y}_t(\bX_t^\infty))\diff B_t^Y,
        \end{cases}
    \end{align}
    with the initial condition $\bX^\infty_0=\bX_0$ and $\bY^\infty_0=\bY_0$. Using the same arguments as in Step 1, by standard SDE theory there exists a strong solution $(\bX^\infty,\bY^\infty)$ with
     \begin{align}\label{moment1}
         \EE\left[\sup_{t\in[0,T]}|\bX_t^\infty|^p+|\bY_t^\infty|^p\right]\leq C.
     \end{align}
     The goal of the next steps is to show that
     \begin{align}
         \label{eq:goal_to_show_eq_func}
         \mathcal{X}=\conx_{\alpha,\beta}(\dmf^X,\dmf^Y)\quad\text{and}\quad\mathcal{Y}(\cdot)=\cony_\beta(\dmf^Y,\cdot),
     \end{align}
     where $\dmf_t^X:=\law(\bX_t^\infty)$ and $\dmf_t^Y:=\law(\bY_t^\infty)$. Consequently, we prove that $(\bX^\infty,\bY^\infty)$ is a solution to \eqref{eq:mf_cbo}. For this we will define for each $R>0$ a stopping time
     \begin{align*}
        \tau_R:=\inf\{t\in [0,T]\;\vert\; |X_t^\infty|\geq R\},
    \end{align*}
    which is adapted to the natural filtration.

    \noindent\underline{\textit{Step 2.1: Proving that
    $\mathcal{Y}=\cony_\beta(\dmf^Y,\cdot)$.}} Define $Z^{Y,n}:=\bY^\infty-\bY^n$, then the process satisfies
        \begin{align*}
            Z_{t\wedge\tau_R}^{Y,n}&=Z_0^{Y,n}-\lambda \int_0^{t\wedge\tau_R} Z_s^{Y,n}\diff s
            +\lambda \int_0^{t\wedge\tau_R} \mathcal{Y}_s(\bX_s^\infty)-\cony_{\beta}(\dmf_s^{X,n},P_n(\bX_s^n))\diff s\\
            &\phantom{=}+\sigma \int_0^{t\wedge\tau_R} D(\bY_s^\infty - \mathcal{Y}_s(\bX_s^\infty))-D(\bY_s^n - \cony_{\beta}(\dmf_s^{X,n},P_n(\bX_s^n))\diff B_s^{Y}
        \end{align*}
        With Jensen's inequality, the identity $\bY_0^\infty=\bY_0^n=\bY_0$, and BDG inequality, one can obtain
    \begin{align*}
        &\EE\left[\sup_{s\in [0,t\wedge \tau_R]}|Z_s^{Y,n}|^2 \right]\leq C\,\EE\left[\int_0^{t\wedge\tau_R} |Z_s^{Y,n}|^2\diff s\right]+ C\,\EE\left[\int_0^{t\wedge\tau_R} |\mathcal{Y}_s(\bX_s^\infty)-\cony_{\beta}(\dmf_s^{X,n},P_n(\bX_s^n))|^2\diff s\right]\,.
    \end{align*}
    Now, if we define $k:=s\wedge\tau_R$ and let $n\geq R$, then it follows from \eqref{LipY} that in the event $k>0$
  \begin{align*}
      &|\mathcal{Y}_k(\bX_k^\infty)-\cony_{\beta}(\dmf_k^{X,n},P_n(\bX_k^n))|= |\mathcal{Y}_k(\bX_k^\infty)-\cony_{\beta}(\dmf_k^{X,n},\bX_k^n)| \\
      &\leq |\mathcal{Y}_k(\bX_k^\infty)-\mathcal{Y}_k(\bX_k^n)|+|\mathcal{Y}_k(\bX_k^n)-\cony_{\beta}(\dmf_k^{X,n},\bX_k^n)|\\
      &\leq  C(R)|Z_k^{Y,n}|+\sup_{x\in B_R}|\mathcal{Y}_k(x)-\cony_{\beta}(\dmf_k^{X,n},x)|\leq C(R)|Z_k^{Y,n}|+\varepsilon_n
  \end{align*}
  holds for some $\varepsilon_n>0$ satisfying $\varepsilon_n\to 0$ as $n\to\infty$, where in the last inequality we have used \eqref{convergenceY}. Thus we have
  \begin{align*}
      &\EE\left[\sup_{s\in [0,t\wedge \tau_R]}|Z_s^{Y,n}|^2 \right]\leq C(R)\,\EE\left[\int_0^{t\wedge \tau_R} |Z_s^{Y,n}|^2\diff s+\varepsilon_n\right]\leq C(R)\left(\int_0^{t} \EE\left[\sup_{r\in [0,s\wedge \tau_R]}|Z_r^{Y,n}|^2\right] \diff s+\varepsilon_n\right),
  \end{align*}
which by the Gronwall inequality leads to
\begin{equation}
     \EE\left[\sup_{s\in [0,t\wedge \tau_R]}|Z_s^{Y,n}|^2 \right]\leq C(R)\,\varepsilon_n.
\end{equation}
Next, we have the estimate
\begin{align*}
    W_2^2(\bar \rho_t^Y,\bar\rho_t^{Y,n})&\leq \EE\left[\left.|Z_t^{Y,n}|^2\;\right\vert\;t\leq \tau_R\right]\mathbb{P}(t\leq \tau_R)+\EE\left[\left.|Z_t^{Y,n}|^2\;\right\vert\;t>\tau_R\right]\mathbb{P}(t>\tau_R)\\
    &\leq \EE\left[\sup_{s\in [0,t\wedge\tau_R]}|Z_s^{Y,n}|^2\right]    +\EE\left[\left.|Z_t^{Y,n}|^2\;\right\vert\;t>\tau_R\right]\mathbb{P}(t>\tau_R).
\end{align*}
Then, let us use the fact that $p>2$, the Chebyshev inequality, and Hölder inequality to show
\begin{align*}
    &\EE\left[\left.|Z_t^{Y,n}|^2\;\right\vert\;t>\tau_R\right]\mathbb{P}(t>\tau_R)
    \leq \frac{\EE\left[|Z_t^{Y,n}|^2 \;\sup\limits_{\tau \in [0,t]}|\bX_\tau^\infty|^{p-2}\right]}{R^{p-2}}\\
  &\leq  C\frac{\EE\left[ \sup\limits_{\tau \in [0,t]}|\bY_\tau^\infty|^p+|\bY_\tau^n|^{p} \right]^{\frac{1}{p}}\EE\left[ \sup\limits_{\tau \in [0,t]}|\bX_\tau^\infty|^p\right]^{\frac{p-2}{p}}}{R^{p-2}}\leq\frac{C}{R^{p-2}}\,,
\end{align*}
where we have used the uniform moment estimates from \ref{prop:uniform_p_mom_R} and \eqref{moment1}, and $C>0$ is independent of $n$ and $R$. This implies that
\begin{equation}\label{WY}
    W_2^2(\bar \rho_t^Y,\bar\rho_t^{Y,n})\leq C(R)\varepsilon_n+\frac{C}{R^{p-2}}\to 0
\end{equation}
by simultaneously letting $n\to \infty$ and then $R\to\infty$.  Finally, for any $x\in\RR^{d_1}$, there exists some $r>0$ such that $x\in B_r$, then according to \eqref{convergenceY} and \eqref{WY} we have
\begin{align*}
   & |\mathcal{Y}_t(x)- \cony_\beta(\dmf_t^Y,x)|\leq |\mathcal{Y}_t(x)-\cony_\beta(\dmf_t^{Y,n},x)|+|\cony_\beta(\dmf_t^{Y,n},x)-\cony_\beta(\dmf_t^Y,x)|\\
   &\leq \sup_{z\in B_r}|\mathcal{Y}_t(z)-\cony_\beta(\dmf_t^{Y,n},z)|+C(r)W_2(\bar \rho_t^Y,\bar\rho_t^{Y,n}) \to 0 \;\mbox{ as }n\to\infty,
\end{align*}
where in the last inequality we have used \ref{est:y_msr_bas} of Lemma \ref{lem:stab_est_y}. Hence we have obtained $\mathcal{Y}(\cdot)=\cony_\beta(\dmf^Y,\cdot)$.

\noindent\underline{\textit{Step 2.2: Proving that
    $\mathcal{X}=\conx_{\alpha,\beta}(\dmf^X,\dmf^Y)$.}} Define $Z^{X,n}:=\bX^\infty-\bX^n$, then similarly as in Step 2.1 we obtain
    \begin{align*}
        &\EE\left[\sup_{s\in [0,t\wedge\tau_R]}|Z_t^{X,n}|^2 \right]\leq C\,\EE\left[\int_0^{t\wedge\tau_R}|Z_s^{X,n}|^2\diff s\right]
            + C\,\EE\left[\int_0^{t\wedge\tau_r} |\mathcal{X}_s-\conx_{\alpha,\beta}(\dmf_s^{X,n},\dmf_s^{Y,n})|^2\diff s\right]\,.
    \end{align*}
    Using \eqref{convergenceX}, there exists $\varepsilon_n>0$ such that $\varepsilon_n\to 0$ as $n\to 0$ with
    \begin{align*}
        &\EE\left[\sup_{s\in [0,t\wedge\tau_R]}|Z_t^{X,n}|^2 \right]
        \leq C\,\int_0^t\EE\left[\sup_{r\in [0,s\wedge\tau_R]}|Z_r^{X,n}|^2 \right]\diff s+C\,\varepsilon_n.
    \end{align*}
    Then, it follows from Gronwall's inequality, the definition of Wasserstein distance that
\begin{equation}
    \label{eq:WX}
   \lim_{n\to\infty} W_2^2(\bar \rho_t^X,\bar\rho_t^{X,n})\leq \lim_{n\to\infty} \EE|Z_t^{X,n}|^2 =0.
\end{equation}
Now, by using Corollary \ref{cor:stab_est_conx}, and \eqref{convergenceX}, \eqref{WY} and \eqref{eq:WX} we deduce the equality
\begin{align*}
  &|\mathcal{X}_t-\conx_{\alpha,\beta}(\dmf_t^X,\dmf_t^Y)|\leq |\mathcal{X}_t-\conx_{\alpha,\beta}(\dmf_t^{X_n} ,\dmf_s^{Y,n})|+|\conx_{\alpha,\beta}(\dmf_t^{X,n} ,\dmf_s^{Y,n})-\conx_{\alpha,\beta}(\dmf_t^X,\dmf_t^Y)|\\
  &\leq |\mathcal{X}_t-\conx_{\alpha,\beta}(\dmf_t^{X,n} ,\dmf_t^{Y,n})|+CW_2(\bar \rho_t^X,\bar\rho_t^{X,n})+CW_2(\bar \rho_t^Y,\bar\rho_t^{Y,n})\to 0\; \mbox{ as }n\to\infty.
\end{align*}
Thus we have shown \eqref{eq:goal_to_show_eq_func} and obtained the existence of a solution to \eqref{eq:mf_cbo}.

\noindent\underline{\textit{Step 3: Uniqueness of solution \eqref{eq:mf_cbo}.}} We prove the uniqueness in this last step. Let us assume $(\bX,\bY)$ and $(\bX',\bY')$ are two solutions to \eqref{eq:mf_cbo} with the same initial data. We define the stopping time 
\begin{align*}
    \tau_R:=\inf\{t\in [0,T]\;\vert\; \min\{|X_t|,|X_t'|\}\geq R\},
\end{align*}
which is adapted to the natural filtration. Define $Z_t^{X}:=\bX_t-\bX_T'$ and $Z_t^{Y}:=\bY_t-\bY_t'$. Then using the same methods as in Step 2.1 and 2.2, we deduce
 \begin{align}\label{ZXes}
        &\EE\left[\sup_{s\in [0,t\wedge \tau_R]}Z_s^{X}\right]^2\leq C\,\EE\left[\int_0^{t\wedge \tau_R} \EE|Z_s^{X}|^2\diff s\right]
            + C\,\EE\left[\int_0^{t\wedge \tau_R} |\conx_{\alpha,\beta}(\dmf_s^{X},\dmf_s^{Y})-\conx_{\alpha,\beta}(\dmf_s^{X'},\dmf_s^{Y'})|^2 \diff s\right]\notag\\
            &\leq C\,\int_0^{t} \EE\left[\sup_{r\in [0,s\wedge \tau_R]}Z_r^{X}\right]^2+\EE\left[\sup_{r\in [0,s\wedge \tau_R]}Z_r^{Y}\right]^2\diff s,
    \end{align}
where in the last inequality we have used the Corollary \ref{cor:stab_est_conx}. Indeed, we estimate
\begin{align*}
   |\conx_{\alpha,\beta}(\dmf_s^{X},\dmf_s^{Y})-\conx_{\alpha,\beta}(\dmf_s^{X'},\dmf_s^{Y'})|^2&\leq  C(W_2^2(\dmf_s^{X},\dmf_s^{X'})+W_2^2(\dmf_s^{Y},\dmf_s^{Y'}))\\&\leq C (\EE|Z_s^X|^2+\EE|Z_s^Y|^2).
\end{align*}
Correspondingly, we can also obtain  
\begin{align}
        \label{eq:zy_estimate}
        \begin{split}
            \EE\left[\sup_{s\in [0,t\wedge \tau_R]}|Z_s^{Y}|^2\right]
        &\leq C\;\EE\left[\int_0^{t\wedge \tau_R} \EE|Z_s^{Y}|^2\diff s\right]\\
        &\phantom{\leq}+ C\,\EE\left[\int_0^{t\wedge \tau_R} \EE|\cony_{\beta}(\dmf_s^{Y},\bX_s)-\cony_{\beta}(\dmf_s^{Y'},\bX_s')|^2\diff s\right]\,.
        \end{split}
    \end{align}
    According to Corollary \ref{cor:stab_est_cony}, if we define $k:=s\wedge \tau_R$, then it holds that
      \begin{align}
        \label{eq:cony_bnd_ins_exc}
          \begin{split}
              &\EE|\cony_{\beta}(\dmf_k^{Y},\bX_k)-\cony_{\beta}(\dmf_k^{Y'},\bX_k'))|^2\leq C(R)\,\big(W_2^2(\dmf_k^{Y},\dmf_k^{Y'})+\EE|\bX_k-\bX_k'|^2\big)\\
              &\leq C(R)\,(\EE|Z_k^X|^2+\EE|Z_k^Y|^2).
          \end{split}
      \end{align}
    Using this fact gives us
  \begin{equation*}
       \EE\left[\sup_{s\in [0,t\wedge \tau_R]}|Z_s^{Y}|^2\right]\leq C(R)\int_0^{t} \EE\left[\sup_{r\in [0,s\wedge \tau_R]}Z_r^{X}\right]^2+\EE\left[\sup_{r\in [0,s\wedge \tau_R]}Z_r^{Y}\right]^2\diff s
  \end{equation*}
  We take the sum of both equations and find with the Gronwall inequality and the fact that $Z_0^X=0$ and $Z_0^Y=0$
  \begin{align*}
      \EE\left[\sup_{t\in [0,T\wedge \tau_R]}|Z_t^X|^2+|Z_t^{Y}|^2\right]=0.
  \end{align*}
  Finally, we need to remove the localization. We have
  \begin{align*}
      \mathbb{P}(\tau_R<T)
        &\leq \mathbb{P}\left(\sup_{t\in [0,T]}|X_t|>R\right)+
      \mathbb{P}\left(\sup_{t\in [0,T]}|X_t'|>R\right)\\&\leq \frac{1}{R^p}\left(\EE\left[\sup_{t\in [0,T]}|X_t|^p\right]+\EE\left[\sup_{t\in [0,T]}|X_t|^p\right]\right).
  \end{align*}
Then by Lemma \ref{lem:moment_estimate_mf_cbo} we have that $\mathbb{P}(\tau_R<T)\to 0$ as $R\to \infty$. Therefore, the two solutions $(X,Y)$ and $(X',Y')$ agree almost everywhere on the interval $[0,T]$. This shows that the solution to \eqref{eq:mf_cbo} is unique. With this we conclude the proof.

    \section{Proof of Theorem \ref{theorem:meanfield_of_cbo}}

    \label{section:meanfield_limit_for_cbo}
    While our strategy is inspired by the mean-field convergence frameworks of
    \cite[Theorem 2.6]{gerber_hoffmann_vaes_2025_meanfield} or \cite[Theorem 3.1]{chaintron_diez_2022_propagation}, , our proof requires substantial new ingredients to handle the coupled consensus points $\conx_{\alpha,\beta}$ and $\cony_\beta$.
    \begin{proof}[Proof of Theorem \ref{theorem:meanfield_of_cbo}]

    Let us explain the proof strategy: The goal of the proof is to devise a Gronwall type estimate between the expected particle difference $|\bX_t^i-X_t^i|^p$ and $|\bY_t^i-Y_t^i|^p$. For this we make use of the SDE's \eqref{eq:cbo} and \eqref{eq:mf_cbo} and bound the difference in consensus between $\dcbo^{m,N}$ and $\overline{\mu}^{m,N}$, and $\overline{\mu}^{m,N}$ and $\dmf^m$, where $\overline{\mu}^{m,N}$ the empirical distribution of \eqref{eq:mf_cbo} and $m\in \{X,Y\}$. We divide the proof into three steps. Using the exact same justification as in \cite[Theorem 2.6]{gerber_hoffmann_vaes_2025_meanfield}, it suffices to prove the claim for fixed $p\in [s+2,\, \frac{q}{2}]$.

        \smallskip
    
        \noindent\underline{\it Coupling Method for the parameter $X$.} We define the particle distribution of \eqref{eq:mf_cbo} as
        \begin{align*}
            \dpmf^{X,N}:=\frac{1}{N}\sum_{i=1}^N \delta_{\bX^i},
            \qquad \dpmf^{N,Y}:=\frac{1}{N}\sum_{i=1}^N \delta_{\bY^i}.
        \end{align*}
        Then by the BDG inequality we derive the estimate
        \begin{align*}
            \EE\left[\sup_{t\in [0,T]}\left\vert X_t^i-\bX_t^i \right\vert^p\right]
            &\leq (2T)^{p-1}\lambda^p\int_0^T \EE\vert X_t^i-\bX_t^i+\conx_{\alpha,\beta}(\dcbo_t^{X,N}, \dcbo_t^{Y,N})-\conx_{\alpha,\beta}(\dmf_t^X, \dmf_t^Y) \vert^p\diff t\\
            &\phantom{\leq}+2^{p-1} T^{\frac{p}{2}-1}\sigma^p C_{\mathrm{DBG}} \int_0^T \EE\left\vert Q_t \right\vert^p\diff t.
        \end{align*}
        By using the inequality $(a+b+c)^p\leq 3^{p-1} (a^p+b^p+c^p)$, we split the RHS into terms that we can control
        \begin{align}
            \EE\left[\sup_{t\in [0,T]}\left\vert X_t^i-\bX_t^i \right\vert^p\right]\tag{P Diff X}\label{eq:mean-field limit:particle diff}
            &\leq (6T)^{p-1}\lambda^p\int_0^T \EE\vert X_t^i-\bX_t^i \vert^p\diff t\\
            &\phantom{\leq} \tag{Emp Diff X}\label{eq:mean-field limit:empirical_difference}
            + (6T)^{p-1}\lambda^p\int_0^T \EE\left\vert A_t \right\vert^p\diff t\\
            &\phantom{\leq} \tag{Emp App X}\label{eq:mean-field limit:empirical_approximation}
            + (6T)^{p-1}\lambda^p\int_0^T \EE\left\vert B_t \right\vert^p\diff t\\
        &\phantom{\leq} \tag{L X}\label{eq:mean-field limit:Lipschitz transform}
        +2^{p-1} T^{\frac{p}{2}-1}\sigma^p C_{\mathrm{DBG}} \int_0^T \EE\big\vert Q_t\big\vert^p\diff t,
        \end{align}
        where we define the differences
        \begin{align*}
            A_t&:=\conx_{\alpha,\beta}(\dcbo_t^{X,N}, \dcbo_t^{Y,N})-\conx_{\alpha,\beta}(\dpmf_t^{X,N}, \dpmf_t^{Y,N}),\\
            B_t&:=\conx_{\alpha,\beta}(\dpmf_t^{X,N}, \dpmf_t^{Y,N})-\conx_{\alpha,\beta}(\dmf_t^X, \dmf_t^Y),\\
            Q_t&:=D(X_t^i-\conx_{\alpha,\beta}(\dcbo_t^{X,N}, \dcbo_t^{Y,N})-D(\bX_t^i-\conx_{\alpha,\beta}(\dmf_t^X, \dmf_t^Y)).
        \end{align*}
        The next step is to bound each term in the above inequality. First, we start with the empirical difference \eqref{eq:mean-field limit:empirical_difference}. For this we define the $\RR$-valued random variables and value
        \begin{align*}
        	\overline{Z}^i
        	:=\sup_{t\in[0,T]} |(\bX_t^i,\bY_t^i)|^p,\qquad
        	R_0:=\EE[
          \overline{Z}^1].
        \end{align*}
        Let $R>R_0$ be fixed, then we need to define the particle excursion set
        \begin{align*}
        	\Omega_t:=\left\{\omega\in\Omega\;\left\vert\;\frac{1}{N}\sum_{i=1}^N |(\bX_t^i(\omega),\bY_t^i(\omega))|^p\geq R\right.\right\}.
        \end{align*}
        Additionally, by Theorem \ref{theorem:existence_and_uniqueness_for_mfcbo} we know that $\EE[|\overline{Z}^{m,i}|^{q/p}]<\infty$. As $q\geq 2p$, we have by \cite[Lemma 2.5]{gerber_hoffmann_vaes_2025_meanfield} that there exists a constant $C>0$ independent of $N$ such that 
        \begin{align}
            \label{eq:prob_of_excursion}
            \mathbb{P}(\Omega_t)\leq C \, N^{-\frac{q}{2p}}.
        \end{align}
        Now we can split the term in \eqref{eq:mean-field limit:empirical_difference} by
        \begin{align*}
        	\EE\left\vert A_t \right\vert^p
        	&=\EE\left[\left\vert A_t \right\vert^p \bbone_{\Omega\setminus\Omega_t}\right]+\EE\left[\left\vert A_t \right\vert^p\bbone_{\Omega_t}\right].
        \end{align*}
        By Corollary \ref{cor:stab_est_conx}, we obtain
        \begin{align*}
        	\EE\left[\left\vert A_t \right\vert^p \bbone_{\Omega\setminus\Omega_t}\right]
        	\leq C\, \EE\left[W_p^p(\dcbo_t^{X,N},\dpmf_t^{X,N})+W_p^p(\dcbo_t^{Y,N},\dpmf_t^{Y,N})\right].
        \end{align*}
        Additionally, the Wasserstein difference is controlled by the particle difference by
        \begin{align*}
        	\EE\left[W_p^p(\dcbo_t^{X,N},\dpmf_t^{X,N})\right]
        	\leq \EE\left[\frac{1}{N}\sum_{j=1}^N |X_t^j-\bX_t^j|^p\right]
        	=\EE|X_t^i-\bX_t^i|^p,   \end{align*}
        and we argue similarly for $W_p^p(\dcbo_t^{Y,N},\dpmf_t^{Y,N})$. Next, we use \ref{eq:upper_bnd_cons} from Lemma \ref{lem:est_weighted_mean}, Lemma \ref{lemma:moment_estimates_for_the_empirical_measures} and Lemma \ref{lem:moment_estimate_mf_cbo} to derive the estimate
        \begin{align*}
        	\EE|A_t|^q\leq 2^{q-1}\big(\EE|\conx_{\alpha,\beta}(\dcbo_t^{X,N}, \dcbo_t^{Y,N})|^q+\EE|\conx_{\alpha,\beta}(\dpmf_t^{X,N}, \dpmf_t^{Y,N})|^q\big)
            \leq C,
        \end{align*}
        where $C>0$ is a constant independent of $N$. Then, by the Hölder inequality and \eqref{eq:prob_of_excursion}, we find
        \begin{align}
            \label{eq:bound_exercusion_particles}
            \begin{split}
                \EE[|A_t|^p \bbone_{\Omega_t}]\leq \big(\EE|A|^q\big)^{\frac{p}{q}}
        	\EE[\bbone_{\Omega_t}^{\frac{q}{q-p}}]^{\frac{q-p}{q}}
        	\leq C\,N^{-\frac{q-p}{2p}}.
            \end{split}
        \end{align}
        We have thus shown
        \begin{align*}
        	\EE|A_t|^p
        	\leq C\, N^{-\frac{q-p}{2p}}+C\,\Big(\EE|X_t^i-\bX_t^i|^p+\EE|Y_t^i-\bY_t^i|^p\Big).
        \end{align*}
        Next, we bound the empirical approximation \eqref{eq:mean-field limit:empirical_approximation}. As a first step we derive the bound
        \begin{align*}
            \EE| B_t |^p&\leq 2^{p-1}\EE\left[\left\vert \conx_{\alpha,\beta}(\dpmf_t^{X,N},\dpmf_t^{Y,N})-\conx_{\alpha,\beta}(\dmf_t^X ,\dpmf_t^{Y,N}) \right\vert^p\mathds{1}_{\Omega_t}\right]\\
            &\phantom{\leq}+2^{p-1}\EE\left[\left\vert \conx_{\alpha,\beta}(\dpmf_t^{X,N},\dpmf_t^{Y,N})-\conx_{\alpha,\beta}(\dmf_t^X ,\dpmf_t^{Y,N}) \right\vert^p\mathds{1}_{\Omega\setminus\Omega_t}\right]\\
            &\phantom{\leq}+2^{p-1}\EE\left[\left\vert \conx_{\alpha,\beta}(\dmf_t^X ,\dpmf_t^{Y,N})-\conx_{\alpha,\beta}(\dmf_t^X ,\dmf_t^Y ) \right\vert^p\right].
        \end{align*}
        Thus, by Lemma \ref{lemma:convergence_of_the_weighted_mean_for_iid_samples} for the last two terms on the RHS there exists some constant $C>0$ independent of $N$ such that
        \begin{align*}
        	&\EE\left[\left\vert \conx_{\alpha,\beta}(\dpmf_t^{X,N},\dpmf_t^{Y,N})-\conx_{\alpha,\beta}(\dmf_t^X ,\dpmf_t^{Y,N}) \right\vert^p\mathds{1}_{\Omega\setminus\Omega_t}\right]\\
            &+\EE\left\vert \conx_{\alpha,\beta}(\dmf_t^X ,\dpmf_t^{Y,N})-\conx_{\alpha,\beta}(\dmf_t^X ,\dmf_t^Y ) \right\vert^p\leq CN^{-\frac{p}{2}}.
        \end{align*}
        Additionally, for the first term on the RHS we can use the same bound as in  \eqref{eq:bound_exercusion_particles} to show
        \begin{align*}
            \EE\left[\left\vert \conx_{\alpha,\beta}(\dpmf_t^{X,N},\dpmf_t^{Y,N})-\conx_{\alpha,\beta}(\dmf_t^X ,\dpmf_t^{Y,N}) \right\vert^p\mathds{1}_{\Omega_t}\right]
            \leq C\, N^{-\frac{q-p}{2p}},
        \end{align*}
        for some constant $C>0$ independent of $N$. Finally, using the fact that $D$ is globally Lipschitz, we can bound the term \eqref{eq:mean-field limit:Lipschitz transform} analogously as we have done above. This means we derive the inequality 
        \begin{align}
            \label{eq:gronwal_est_x}
        	\begin{split}
        	    &\EE\left[\sup_{t\in [0,T]}\left\vert X_t^i-\bX_t^i \right\vert^p\right]\leq CN^{-\theta p}+C\int_0^T \EE\left[\sup_{s\in [0,t]}\left\vert X_s^i-\bX_s^i \right\vert^p+\left\vert Y_s^i-\bY_s^i \right\vert^p\right]\diff t
        	\end{split}
        \end{align}
        and we define the parameter
        \begin{align*}
            \theta:=\min\left\{\frac{1}{2},\; \frac{q-p}{2p^2}\right\}
            =\min\left\{\frac{1}{2},\; \frac{q-p}{2p^2},\; \frac{q-(s+2)}{2 (s+2)^2}\right\},
        \end{align*}
        where we used $(0,q)\ni z\mapsto \frac{q-z}{z^2}\in \RR$ is decreasing. 

        \smallskip
        
        \underline{\it Coupling Method for the parameter $Y$.} We apply the same arguments as in the previous step. We simply highlight key differences. It suffices to derive bounds for the following two terms
        \begin{align}
            \tag{Emp Diff Y}\label{eq:mean-field limit:empirical_difference_y}
            &\EE \big\vert \cony_\beta(\dcbo_t^{Y,N}, X_t^i)-\cony_\beta(\dpmf_t^{Y,N},\bX_t^i) \big\vert^p,\\
            \tag{Emp App Y}\label{eq:mean-field limit:empirical_approximation_y}
            &\EE\big\vert \cony_\beta(\dpmf_t^{Y,N},\bX_t^i)-\cony_\beta(\dmf_t^Y,\bX_t^i) \big\vert^p.
        \end{align}
        For \eqref{eq:mean-field limit:empirical_difference_y}, we use the same excursion set trick to rewrite the expectation as
        \begin{align*}
        	&\EE\big\vert \cony_\beta(\dcbo_t^{Y,N}, X_t^i)-\cony_\beta(\dpmf_t^{Y,N},\bX_t^i) \big\vert^p\\
        	&=\EE\left[\left\vert \cony_\beta(\dcbo_t^{Y,N}, X_t^i)-\cony_\beta(\dpmf_t^{Y,N},\bX_t^i) \right\vert^p \bbone_{\Omega\setminus\Omega_t}\right]\\
        	&\phantom{=}+\EE\left[\left\vert \cony_\beta(\dcbo_t^{Y,N}, X_t^i)-\cony_\beta(\dpmf_t^{Y,N},\bX_t^i) \right\vert^p\bbone_{\Omega_t}\right].
        \end{align*}
        By the inequality \ref{est:y_var_msr} of Corollary \ref{cor:stab_est_cony}, we obtain
        \begin{align*}
        	\EE\left[\left\vert \cony_\beta(\dcbo_t^{Y,N}, X_t^i)-\cony_\beta(\dpmf_t^{Y,N},\bX_t^i) \right\vert^p \bbone_{\Omega\setminus\Omega_t}\right]
        	\leq C\,\left( \EE\left[W_p^p(\dcbo_t^{Y,N},\dpmf_t^{Y,N})\right]+\EE|X_t^i-\bX_t^i|^p\right).
        \end{align*}
        At this point, we can follow the same steps as previously, to derive the inequality
        \begin{align*}
        	&\EE|\cony_\beta(\dcbo_t^{X,N}, \dcbo_t^{Y,N})-\cony_\beta(\dpmf_t^{X,N}, \dpmf_t^{Y,N})|^p\leq CN^{-\frac{q-p}{2p}}+C\EE|X_t^i-\bX_t^i|^p+\EE|Y_t^i-\bY_t^i|^p.
        \end{align*}
        For \eqref{eq:mean-field limit:empirical_approximation_y} we use Lemma \ref{lemma:convergence_of_the_weighted_mean_for_iid_samples_y}  to show there exists a constant $C>0$ independent of $N$ such that
        \begin{align*}
            \EE\big\vert \cony_\beta(\dpmf_t^{Y,N},\bX_t^i)-\cony_\beta(\dmf_t^Y,\bX_t^i) \big\vert^p
            \leq CN^{-\frac{p}{2}},
        \end{align*}
        If we combine everything we have proven the inequality
        \begin{align}
            \label{eq:gronwal_est_y}
            \begin{split}
                &\EE\left[\sup_{t\in [0,T]}\left\vert Y_t^i-\bY_t^i \right\vert^p\right]\leq CN^{-\theta p}+C\int_0^T \EE\left[\sup_{s\in [0,t]}\left\vert X_s^i-\bX_s^i \right\vert^p+\left\vert Y_s^i-\bY_s^i \right\vert^p\right]\diff t
            \end{split}
        \end{align}

        \smallskip
        
        \underline{\it Comining both inequalities.}
        We take the sum of \eqref{eq:gronwal_est_x} and \eqref{eq:gronwal_est_y} and find
        \begin{align*}
        	&\EE\left[\sup_{t\in [0,T]}\left\vert X_t^i-\bX_t^i \right\vert^p+\left\vert Y_t^i-\bX_t^i \right\vert^p\right]\leq CN^{-\theta p}+C\int_0^T \EE\left[\sup_{s\in [0,t]}\left\vert X_s^i-\bY_s^i \right\vert^p+\left\vert Y_s^i-\bY_s^i \right\vert^p\right]\diff t.
        \end{align*}
        The result is now a consequence of the Grönwall inequality and taking the $p$th root. With this we conclude the proof.
    \end{proof}
    
\appendix

\bibliography{references.bib} 

\begin{thebibliography}{10}

\bibitem{bayraktar_ekren_zhou_2025_uniformintime}
Erhan Bayraktar, Ibrahim Ekren, and Hongyi Zhou.
\newblock Uniform-in-time weak propagation of chaos for consensus-based
  optimization.
\newblock {\em arXiv preprint arXiv:2502.00582v1}, 2025.

\bibitem{beddrich_chenchene_fornasier_huang_wohlmuth_2024_constrained}
Jonas Beddrich, Enis Chenchene, Massimo Fornasier, Hui Huang, and Barbara
  Wohlmuth.
\newblock Constrained consensus-based optimization and numerical heuristics for
  the few particle regime.
\newblock {\em arXiv preprint arXiv:2410.10361v1}, 2024.

\bibitem{oksendal_2003_stochastic}
{Bernt {\O}ksendal}.
\newblock {\em Stochastic differential equations: An introduction with
  applications}.
\newblock Universitext. Springer, 2003.

\bibitem{borghi_herty_pareschi_2023_adaptive}
Giacomo Borghi, Michael Herty, and Lorenzo Pareschi.
\newblock An adaptive consensus based method for multi-objective optimization
  with uniform {Pareto} front approximation.
\newblock {\em Applied Mathematics \& Optimization}, 88(2):58, 2023.

\bibitem{borghi_huang_qiu_2024_particle}
Giacomo Borghi, Hui Huang, and Jinniao Qiu.
\newblock A particle consensus approach to solving nonconvex-nonconcave min-max
  problems.
\newblock {\em SIAM Journal on Control and Optimization}, 2026, to appear.

\bibitem{bungert_hoffmann_kim_roith_2025_mirrorcbo}
Leon Bungert, Franca Hoffmann, Dohyeon Kim, and Tim Roith.
\newblock {MirrorCBO}: A consensus-based optimization method in the spirit of
  mirror descent.
\newblock {\em Mathematical Models and Methods in Applied Sciences},
  35(14):3083--3170, 2025.

\bibitem{byeon_ha_hwang_ko_yoon_2025_consensus}
Junhyeok Byeon, Seung~Yeal Ha, Gyuyoung Hwang, Dongnam Ko, and Jaeyoung Yoon.
\newblock Consensus, error estimates and applications of first-and second-order
  consensus-based optimization algorithms.
\newblock {\em Mathematical Models and Methods in Applied Sciences},
  35(2):345--401, 2025.

\bibitem{carrillo_choi_totzeck_tse_2018_analytical}
Jos{\'e}~A. Carrillo, Young-Pil Choi, Claudia Totzeck, and Oliver Tse.
\newblock An analytical framework for consensus-based global optimization
  method.
\newblock {\em Mathematical Models and Methods in Applied Sciences},
  28(06):1037--1066, 2018.

\bibitem{chaintron_diez_2022_propagation}
Louis-Pierre Chaintron and Antoine Diez.
\newblock Propagation of chaos: A review of models, methods and applications.
  {II}. applications.
\newblock {\em Kinetic \& Related Models}, 15(6), 2022.

\bibitem{chenchene_huang_qiu_2025_consensusbased}
Enis Chenchene, Hui Huang, and Jinniao Qiu.
\newblock A consensus-based algorithm for non-convex multiplayer games.
\newblock {\em Journal of Optimization Theory and Applications}, 206(2):45,
  2025.

\bibitem{choi2025modified}
Young-Pil Choi, Seungchan Lee, and Sihyun Song.
\newblock A modified consensus-based optimization model: consensus formation
  and uniform-in-time propagation of chaos.
\newblock {\em arXiv preprint arXiv:2511.19116}, 2025.

\bibitem{chow_teicher_1997_probability}
Yuan~Shih Chow and Henry Teicher.
\newblock {\em Probability theory}.
\newblock Springer Texts in Statistics. Springer, New York, 1997.
\newblock ISSN: 1431-875X.

\bibitem{doukhan_lang_2009_evaluation}
Paul Doukhan and Gabriel Lang.
\newblock Evaluation for moments of a ratio with application to regression
  estimation.
\newblock {\em Bernoulli}, 15(4):1259--1286, 2009.

\bibitem{fornasier_huang_pareschi_sunnen_2020_consensusbased}
Massimo Fornasier, Hui Huang, Lorenzo Pareschi, and Philippe S{\"u}nnen.
\newblock Consensus-based optimization on hypersurfaces: Well-posedness and
  mean-field limit.
\newblock {\em Mathematical Models and Methods in Applied Sciences},
  30(14):2725--2751, 2020.

\bibitem{fornasier_klock_riedl_2024_consensusbased}
Massimo Fornasier, Timo Klock, and Konstantin Riedl.
\newblock Consensus-based optimization methods converge globally.
\newblock {\em SIAM Journal on Optimization}, 34(3):2973--3004, 2024.

\bibitem{fornasier2025pde}
Massimo Fornasier and Lukang Sun.
\newblock A pde framework of consensus-based optimization for objectives with
  multiple global minimizers.
\newblock {\em Communications in Partial Differential Equations},
  50(4):493--541, 2025.

\bibitem{trillos_kumarakash_li_riedl_zhu_2025_defending}
Nicol{\'a}s Garc{\'\i}a~Trillos, Aditya Kumar~Akash, Sixu Li, Konstantin Riedl,
  and Yuhua Zhu.
\newblock Defending against diverse attacks in federated learning through
  consensus-based bi-level optimization.
\newblock {\em Philosophical Transactions A}, 383(2298):20240235, 2025.

\bibitem{gerber_hoffmann_kim_vaes_2025_uniformintime}
Nicolai Gerber, Franca Hoffmann, Dohyeon Kim, and Urbain Vaes.
\newblock Uniform-in-time propagation of chaos for consensus-based
  optimization.
\newblock {\em arXiv preprint arXiv:2505.08669v1}, 2025.

\bibitem{gerber_hoffmann_vaes_2025_meanfield}
Nicolai~Jurek Gerber, Franca Hoffmann, and Urbain Vaes.
\newblock Mean-field limits for consensus-based optimization and sampling.
\newblock {\em ESAIM: Control, Optimisation and Calculus of Variations},
  31(74):1--44, 2025.

\bibitem{gilbarg_trudinger_2001_elliptic}
David Gilbarg and Neil~S. Trudinger.
\newblock {\em Elliptic partial differential equations of second order}.
\newblock Classics in mathematics. Springer, Berlin, reprint of 1998 ed.
  edition, 2001.

\bibitem{goodfellow_pougetabadie_mirza_xu_wardefarley_ozair_courville_bengio_2014_generative}
Ian~J Goodfellow, Jean Pouget-Abadie, Mehdi Mirza, Bing Xu, David Warde-Farley,
  Sherjil Ozair, Aaron Courville, and Yoshua Bengio.
\newblock Generative adversarial nets.
\newblock {\em Advances in Neural Information Processing Systems}, 27, 2014.

\bibitem{ha_hwang_kim_2024_timediscrete}
Seung-Yeal Ha, Gyuyoung Hwang, and Sungyoon Kim.
\newblock Time-discrete momentum consensus-based optimization algorithm and its
  application to {Lyapunov} function approximation.
\newblock {\em Mathematical Models and Methods in Applied Sciences},
  34(06):1153--1204, 2024.

\bibitem{ha_kang_kim_kim_yang_2022_stochastic}
Seung-Yeal Ha, Myeongju Kang, Dohyun Kim, Jeongho Kim, and Insoon Yang.
\newblock Stochastic consensus dynamics for nonconvex optimization on the
  {Stiefel} manifold: Mean-field limit and convergence.
\newblock {\em Mathematical Models and Methods in Applied Sciences},
  32(03):533--617, 2022.

\bibitem{herty_huang_kalise_kouhkouh_2025_multiscale}
Michael Herty, Yuyang Huang, Dante Kalise, and Hicham Kouhkouh.
\newblock A multiscale consensus-based algorithm for multilevel optimization.
\newblock {\em Mathematical Models and Methods in Applied Sciences}, pages
  1--37, 2025.

\bibitem{huang_kouhkouh_2025_uniformintime}
Hui Huang and Hicham Kouhkouh.
\newblock Uniform-in-time mean-field limit estimate for the consensus-based
  optimization.
\newblock {\em ESAIM: Control, Optimisation and Calculus of Variations},
  31(69):1--17, 2025.

\bibitem{huang_qiu_2022_meanfield}
Hui Huang and Jinniao Qiu.
\newblock On the mean-field limit for the consensus-based optimization.
\newblock {\em Mathematical Methods in the Applied Sciences},
  45(12):7814--7831, 2022.

\bibitem{huang_qiu_riedl_2023_global}
Hui Huang, Jinniao Qiu, and Konstantin Riedl.
\newblock On the global convergence of particle swarm optimization methods.
\newblock {\em Applied Mathematics \& Optimization}, 88(2):30, 2023.

\bibitem{huang_qiu_riedl_2024_consensusbased}
Hui Huang, Jinniao Qiu, and Konstantin Riedl.
\newblock Consensus-based optimization for saddle point problems.
\newblock {\em SIAM Journal on Control and Optimization}, 62(2):1093--1121,
  2024.

\bibitem{huang_warnett_2025_wellposedness}
Hui Huang and Jethro Warnett.
\newblock Well-posedness and mean-field limit estimate of a consensus-based
  algorithm for multiplayer games.
\newblock {\em Communications on Pure and Applied Analysis}, 2025 to appear.

\bibitem{khasminskii_2012_stochastic}
Rafail Khasminskii.
\newblock {\em Stochastic stability of differential equations}, volume~66 of
  {\em Stochastic Modelling and Applied Probability}.
\newblock Springer, Berlin, Heidelberg, 2012.

\bibitem{ko_ha_jin_kim_2022_convergence}
Dongnam Ko, Seung-Yeal Ha, Shi Jin, and Doheon Kim.
\newblock Convergence analysis of the discrete consensus-based optimization
  algorithm with random batch interactions and heterogeneous noises.
\newblock {\em Mathematical Models and Methods in Applied Sciences},
  32(06):1071--1107, 2022.

\bibitem{mao_2011_stochastic}
Xuerong Mao.
\newblock {\em Stochastic differential equations and applications}.
\newblock Woodhead Publishing, second edition, reprinted edition, 2011.

\bibitem{nash_1950_equilibrium}
John~F. Nash.
\newblock Equilibrium points in \textit{n}-person games.
\newblock {\em Proceedings of the National Academy of Sciences}, 36(1):48--49,
  January 1950.

\bibitem{pinnau_totzeck_tse_martin_2017_consensusbased}
Ren{\'e} Pinnau, Claudia Totzeck, Oliver Tse, and Stephan Martin.
\newblock A consensus-based model for global optimization and its mean-field
  limit.
\newblock {\em Mathematical Models and Methods in Applied Sciences},
  27(01):183--204, 2017.

\bibitem{totzeck_2021_trends}
Claudia Totzeck.
\newblock Trends in consensus-based optimization.
\newblock In {\em Active Particles, Volume 3: Advances in Theory, Models, and
  Applications}, pages 201--226. Springer, 2021.

\bibitem{wang_li_huang_2025_mathematical}
Jinhuan Wang, Keyu Li, and Hui Huang.
\newblock Mathematical analysis of the {PDE} model for the consensus-based
  optimization.
\newblock {\em arXiv preprint arXiv:2504.10990v1}, 2025.

\bibitem{zhang_poupart_yu_2022_optimality}
Guojun Zhang, Pascal Poupart, and Yaoliang Yu.
\newblock Optimality and stability in non-convex smooth games.
\newblock {\em Journal of Machine Learning Research}, 23(35):1--71, 2022.

\end{thebibliography}
\bibliographystyle{plain}

\section{Auxillary Results}

\begin{theorem}[{\cite[Theorem 1]{doukhan_lang_2009_evaluation}}]
    \label{thm:mf_stationary_sequence}
    Let $\{w_j, V_j\}_{j \in \mathbb{N}}$ be a stationary sequence with values in $\mathbb{R} \times \mathbb{R}^d$ with $w_j > 0$ almost surely, and set
    \begin{align*}
        \hat{N}_J = \frac{1}{J} \sum_{j=1}^{J} w_j V_j, \qquad
        \hat{D}_J = \frac{1}{J} \sum_{j=1}^{J} w_j, \qquad
        \hat{R}_J = \frac{\hat{N}_J}{\hat{D}_J}.
    \end{align*}
    Let also $N = \EE[\hat{N}_1]$, $D = \EE[\hat{D}_1]$, and $R = \frac{N}{D}$. Let $0 < p < q$ and assume that for some $\theta,C > 0$,
    \begin{align}
        \label{eq:uniform_bnds_for_stationary_sequence}
            &r:=\frac{p(q+2)}{q-p},\quad
            m:=\frac{pq}{q-p},\quad
            \EE[w_1^q]\leq \theta,\quad
            \EE|V_1|^r\leq \theta, \quad
            \EE|w_1 V_1|^m\leq \theta,\notag\\
            &\Big(\EE|\hat{D}_J-D|^q\Big)^{\frac{1}{q}}
            \leq C J^{-\frac{1}{2}},\qquad
            \Big(\EE|\hat{N}_J-N|^p\Big)^{\frac{1}{p}}
            \leq C J^{-\frac{1}{2}}.
    \end{align}
    Then the following inequality is satisfied 
    \begin{align}
        \label{eq:mf_stationary_sequence}
        \EE|\hat{R}_J - R|^p \leq  \frac{C}{D}\left(1+2\frac{N}{D}+\frac{\theta}{D}+\theta\left(\frac{C}{D}\right)^{2/r}\right)\; J^{-\frac{1}{2}}.
    \end{align}
\end{theorem}

\begin{lemma}[Uniform mean-field convergence]
    \label{lem:unif_emp_msr_conv}
    Let $2\leq p<r$ and let $\{f_i:\RR^d\to\RR\}_{i\in \mathcal{A}}$ denote some family of Borel measurable mappings, let $\mu\in \rP_r(\RR^d)$, and let $V_1,V_2,\ldots$ be i.i.d. samples of $\mu$, and define $\omega_{ij}:= f_i(V_j)$ for $j\in\NN$ and $i\in\mathcal{A}$. Define $q:=\frac{p(r+2)}{r-p}$ and $m:=\frac{pq}{q-p}$. Define
    \begin{align*}
        &\hat{N}_{iJ}:= \frac{1}{J}\sum_{j=1}^J \omega_{ij}V_j,\quad 
        &&\hat{D}_{iJ}:= \frac{1}{J}\sum_{j=1}^J \omega_{ij},\quad
        &&\hat{R}_{iJ}:=\frac{\hat{N}_{iJ}}{\hat{D}_{iJ}},\\
        &N_i := \EE[\hat{N}_{i1}],\quad 
        &&D_i := \EE[\hat{D}_{i1}],\quad 
        &&R_i:=\EE[\hat{R}_{i1}].
    \end{align*}
    Suppose that
    \begin{align}
        \label{eq:unif_emp_msr_conv_cond}
        \sup_{\substack{i\in\mathcal{A} \\ j\in \NN}}\;\EE[\omega_{ij}^q]+\EE|\omega_{ij}V_j|^s+N_i+\frac{1}{D_i}
        <\infty,
    \end{align}
    then there exists a constant $C>0$ independent of $J\in \NN$ such that
    \begin{align*}
        \sup_{i\in\mathcal{A}}\EE[|\hat{R}_{iJ}-R_i|^p]
        \leq C\,J^{-\frac{1}{2}}
    \end{align*}
\end{lemma}
\begin{proof}
        We adapt the proof of \cite[Lemma 3.7]{gerber_hoffmann_vaes_2025_meanfield}. In particular, we apply Theorem \ref{thm:mf_stationary_sequence}. We show that the conditions in \eqref{eq:uniform_bnds_for_stationary_sequence} are satisfied for all $i\in \mathcal{A}$ coefficient, and that the coefficients in \eqref{eq:mf_stationary_sequence} are uniformly bounded for all $i\in\mathcal{A}$. To that end, its clear by assumption that
        \begin{align*}
            \sup_{\substack{i\in\mathcal{A} \\ j\in \NN}} \EE[\omega_{ij}^q]+\EE|V_i|^r+\EE|\omega_{ij}V_j|^m<\infty.
        \end{align*}
        By the Marcinkiewicz–Zygmund inequality \cite[Chapter 10.3, Theorem 2]{chow_teicher_1997_probability} there exists a constant $B_p>0$ depending on $p$ such that
        \begin{align*}
            \EE|\hat{N}_{iJ}-N_i|^p
            &=\frac{1}{J^p}\EE\left\vert \sum_{j=1}^J \omega_{ij} V_j-\EE[\omega_{i1} V_1]\right\vert^p\leq \frac{B_p}{J^p}\EE\left\vert \sum_{j=1}^J |\omega_{ij} V_j-\EE[\omega_{i1} V_1]|^2\right\vert^{\frac{p}{2}}.
        \end{align*}
        Next, we use the Jensen inequality to derive
        \begin{align*}
            &\frac{B_p}{J^p}\EE\left\vert \sum_{j=1}^J |\omega_{ij} V_j-\EE[\omega_{i1} V_1]|^2\right\vert^{\frac{p}{2}}
            \leq \frac{B_p}{J^{\frac{p}{2}+1}}\EE\left[ \sum_{j=1}^J |\omega_{ij} V_j-\EE[\omega_{i1} V_1]|^p\right]\\
            &=\frac{B_p}{J^{\frac{p}{2}}}\EE|\omega_{i1} V_1-\EE[\omega_{i1} V_1]|^p
            \leq 2^p B_p\,\EE |\omega_{i1} V_1|^p\, J^{-\frac{p}{2}}\leq 2^p B_p e^{\alpha p c}\left(\int_{\RR^d} |x|^r\diff \mu(x)\right)^{\frac{p}{r}} J^{-\frac{p}{2}}.
        \end{align*}
        We establish a similar bound for $\EE|\hat{D}_{iJ}-D_i|^q$. By assumption the numerator $D_i$ and denominator are uniformly bounded from above and below respectively. Thus we conclude the proof.
    \end{proof}

\section{Well-posedness of truncated mean-field equation}
\label{section:proof_trunc_sde}
\begin{proof}[Proof of Lemma \ref{lem:existence_trunc_sol}]
We prove that the truncated equation \eqref{eq:mf_cbo1} is well-posed by following the proof of \cite[Theorem 2.4]{gerber_hoffmann_vaes_2025_meanfield}, who in turn follow the proof of \cite[Theorem 3.1, Theorem 3.2]{carrillo_choi_totzeck_tse_2018_analytical}. 

We shall use standard Leray-Schauder fixed point argument. Consequently, the proof is split up into five parts. For simplicity, we drop the sup-script $R$ in the following computation.

\smallskip

\noindent\underline{\it Step 1.1: Solution operator.} Let us define the space of continuous functions $\Theta_R:=C^0([0,T],\RR^{d_1})\times C^0([0,T]\times B_R,\RR^d)$. For some given $(u^X,u^Y)\in \Theta_R$, by using the classical theory of SDEs \cite[Theorem 5.2.1]{oksendal_2003_stochastic} we can uniquely solve the following SDEs,
\begin{align}\label{LCBO}
\begin{cases}
    \diff X_t
    =-\lambda( X_t-u_t^X)\diff t
    +\sigma D( X_t-u_t^X)\diff B_t^{X},\\
    \diff Y_t
    =-\lambda( Y_t-u_t^Y(P_R(X_t)))\diff t
    +\sigma D( Y_t-u_t^Y(P_R(X_t)))\diff B_t^{Y},
\end{cases}
\end{align}
with the same initial condition $X_0=\bX^R_0$ and $Y_0=\bY^R_0$, and the same Brownian motions as in \eqref{eq:mf_cbo1}. We define the laws $\nu_t^m=\mbox{Law} (m_t)$, $m\in \{X,Y\}$. Additionally, the processes $(X,Y)$ are almost surely continuous.  For any $0\leq r\leq t\leq T$ and $m\in \{X,Y\}$, it holds by BDG inequality that
\begin{align}\label{conti}
    \begin{split}
       &\EE\left[\sup_{s\in[r,t]}|m_s-m_r|^p\right] \\ 
    &\leq 2^{p-1}(t-r)^{p-1}\lambda^p \int_r^t\EE|m_s-u_s^m|^p\diff s\\
    &\phantom{\leq}+C_{BDG}\,2^{p-1}\sigma^p(t-r)^{p/2-1}\int_r^t\EE|D(m_s-u_s^m)|^p\diff s.
    \end{split}
\end{align}
Furthermore, letting $r=0$, one obtains
\begin{align*}
    \EE\left[\sup_{s\in[0,t]}|m_s|^p\right]	\leq C\left(\EE|m_0|^p+\|u^m\|_\infty^p+\int_0^t\EE|m_s|^p\diff s\right),
\end{align*}
where $C>0$ depends only on $T$, $\lambda$, $\sigma$ and $p$.  Here we have used the simplified notations
\begin{align*}
    \|u^X\|_\infty:= \|u^X\|_{C^{0}([0,T];\RR^{d_1})}= \sup_{t\in[0,T]} |u_t^X|<\infty.
\end{align*}
and
\begin{align*}
    \|u^Y\|_\infty:= \|u^Y\|_{C^{0}([0,T]\times B_R;\RR^{d_2})}= \sup_{t\in[0,T]}\sup_{x\in B_R} |u_t^Y(x)|<\infty.
\end{align*}
Then, it follows from Gronwall's inequality that
\begin{equation}\label{momentbound}
\EE\left[\sup_{t\in[0,T]}|m_t|^p\right]\leq C\;\big(\EE|m_0|^p+\|u^m\|_\infty^p\big)<\infty.
\end{equation}
Therefore, by using the dominated convergence theorem, we have for any $(t,x)\in[0,T]\times B_R$ that 
\begin{align*}
    \begin{split}
        \cony_\beta(\nu_t^Y,x)=\frac{\int_{\RR^d} y\, \omega_{-\beta}^{\rE}(y,x)\diff\nu_t^Y(y)}{\int_{\RR^d} \omega_{-\beta}^{\rE}(y,x)\diff \nu_t^Y(y)}\to \frac{\int_{\RR^d} y\, \omega_{-\beta}^{\rE}(y,x')\diff \nu_r^Y(y)}{\int_{\RR^d} \omega_{-\beta}^{\rE}(y,x')\diff \nu_r^Y(y)}=	\cony_\beta(\nu_r^Y, x')
    \end{split}
    \quad\text{as } (t,x)\to (r,x')\,.
\end{align*}
and by \ref{eq:upper_bnd_cons} from Lemma \ref{lem:est_weighted_mean}
\begin{align}\label{Ybound}
\begin{split}
    &\| \cony_\beta(\nu^Y,\cdot )\|_\infty
    = \sup_{t\in[0,T]} \sup_{x\in B_R} \left|\frac{\int_{\RR^d} y\, \omega_{-\beta}^{\rE}(y,x)\diff\nu_t^Y(y)}{\int_{\RR^d} \omega_{-\beta}^{\rE}(y,x)\diff \nu_t^Y(y)}\right|
    \leq\sup_{t\in[0,T]}C\;\left(\int_{\RR^d}|y|^p\diff\nu_t^Y(y)\right)^{\frac{1}{p}}<\infty\,.
\end{split}
\end{align}
This implies that the function $(t,x)\mapsto \cony_\beta(\nu_t^Y, x)$ is in $C^0([0,T]\times B_R,\RR^{d_2})$. 
Similarly one can show that $t\mapsto \conx_{\alpha,\beta}(\nu_t^X)$ is in $C^0([0,T],\RR^{d_1})$.
Thus the following solution operator is well-defined
\begin{align*}
    \mathcal{T}:
 \left\{
 \begin{array}{rcl}
      \Theta_R &\to& \Theta_R,  \\
      (u^X,u^Y)&\mapsto &(\conx_{\alpha,\beta}(\nu^X,\nu^Y),\cony_\beta(\nu^Y,\cdot)).
 \end{array}\right.
\end{align*}
\underline{\it Step 1.2: $\mathcal{T}$ is continuous.} Take any two maps $(u^X,u^Y),(u^{X'},u^{Y'})\in \Theta_R$ and let the processes $(X,Y)$, $(X',Y')$ be the corresponding solutions to \eqref{LCBO} with respective laws $\{\nu^m\}_{m\in\{X,Y,X',Y'\}}$. The idea is to first bound the consensus functions with the expected particle difference, and then bound the expected particle difference by the continuous functions $u^X-u^{X'}$ and $u^Y-u^{Y'}$. To that end, we use Corollary \ref{cor:stab_est_conx} to estimate the difference between
    \begin{align}
    \label{eq:computediffX}
    \begin{split}
        &|\conx_{\alpha,\beta}(\nu_t^X,\nu_t^Y)-\conx_{\alpha,\beta}(\nu_t^{X'},\nu_t^{Y'})|^p
    \leq C \,\big(W_p^p(\nu_t^X,\nu_t^{X'})+W_p^p(\nu^Y_t,\nu_t^{Y'})\big)\\
    &\leq C\, \EE\left[\sup_{s\in[0,t]}|X_s-X_s'|^p+|Y_s-Y_s'|^p\right],
    \end{split}
\end{align}
and Corollary  \ref{cor:stab_est_cony} together with the Lipschitz continuity of $P_R$ to estimate the difference between
\begin{align}\label{eq:computediffY}
    \begin{split}
    &\sup_{x\in B_R}|\cony_{\beta}(\nu_t^Y,x)-\cony_{\beta}(\nu_t^{Y'},x)|^p
    \leq C\, W_p^p(\nu_t^Y,\nu_t^{Y'}) \\
    &\leq C\, \EE\left[\sup_{s\in[0,t]}|Y_s-Y_s'|^p\right].
    \end{split}
\end{align}
Moreover, it holds by BDG inequality that
\begin{align*}
        \EE\left[\sup_{s\in[0,t]}|X_s-X_s'|^p\right]&\leq 2^{p-1}t^{p-1}\lambda^p \int_0^t\EE|X_s-X_s'-u_s^X+u_s^{X'}|^p\diff s\\
    &\phantom{\leq}+C_{BDG}\,2^{p-1}\sigma^p(t-r)^{p/2-1}\int_r^t\EE|D(X_s-u_s^X)-D(X_s'-u_s^{X'})|^p\diff s\\
    & \leq C\left(\int_0^t \EE|X_s-X_s'|^p\diff s+\int_0^t \EE|u_s^X-u_s^{X'}|^p\diff s\right).
\end{align*}
Now, by the Gronwall inequality we derive the inequality
\begin{align}
\label{eq:part_diff_cts_x}
\EE\left[\sup_{s\in[0,T]}|X_s-X_s'|^p\right]
    \leq C\, \|u^X-u^{X'}\|_\infty^p.
\end{align}
Similarly, it also holds that
  \begin{align}
  \label{eq:part_diff_cts_y}
\EE\left[\sup_{s\in[0,T]}|Y_s-Y_s'|^p\right]
    \leq C\, \|u^Y-u^{Y'}\|_\infty^p.
\end{align}
Combining \eqref{eq:computediffX} with \eqref{eq:part_diff_cts_x} and \eqref{eq:part_diff_cts_y}, and combining \eqref{eq:computediffY} with \eqref{eq:part_diff_cts_y} proves that $\mathcal{T}$ is a continuous operator.

\noindent\underline{\it Step 1.3: $\mathcal{T}$ is compact.} To show that $\mathcal{T}$ is a compact operator, we fix $M>0$ and define
\begin{equation*}
    \Lambda_M:=\left\{ (u^X,u^Y) \in \Theta_R\;\left\vert\, \max_{m\in \{X,Y\}}\|u^m\|_{\infty}\leq M \right.\right\}.
\end{equation*}
Now, we take any function $(u^X,u^Y) \in \Lambda_M$ and consider the corresponding solution $(X, Y)$ of \eqref{LCBO} with pointwise law $(\nu^X, \nu^Y)$. It has already been shown in \eqref{Ybound} that 
\begin{equation*}
         \|\conx_{\alpha,\beta}(\nu^X,\nu^Y)\|_\infty
         +\|\cony_{\beta}(\nu^Y,\cdot)\|_\infty\leq C\,.
\end{equation*} 
Furthermore, in view of \eqref{conti} and \eqref{momentbound}, there is a constant $L>0$ depending on $M$, $T$, $\lambda$, $\sigma$ and $p$ such that we have the Hölder continuity
\begin{align*}
    \EE|X_t-X_r|^p\leq 2^{-p}L|t-r|^{p/2},\quad \EE|Y_t-Y_r|^p\leq 2^{-p}L|t-r|^{p/2},
\end{align*}
which implies that $W_p(\nu_t^X,\nu_r^X)+W_p(\nu_t^Y,\nu_r^Y)\leq L^{1/p}|t-r|^{1/2}$. 
Then, by Corollary \ref{cor:stab_est_conx}, it holds that
\begin{align}\label{conditionX}
\conx_{\alpha,\beta}(\nu^X,\nu^Y)\in \bigg\{&f\in C^{0}([0,T];\,\RR^{d_1})\;\bigg\vert\; \|f\|_\infty\leq C \notag \\
&\mbox{and }  |f(t)-f(r)|\leq C|t-r|^{1/2},\; \forall~t,r\in[0,T]\bigg\}
\end{align}
and by Corollary \ref{cor:stab_est_conx} one has
\begin{align}\label{conditionY}
\notag\cony_\beta(\nu^Y,\cdot )\in \bigg\{&f\in C^{0}([0,T]\times B_R;\,\RR^{d_2})\;\bigg\vert\;\|f\|_\infty\leq C \\
 &\mbox{and } |f(t,x)-f(r,y)|\leq C(|t-r|^{1/2}+|x-y|),\;\\ &\notag  \forall~t,r\in[0,T], \forall~x,y\in B_R\bigg\}\,.
\end{align}
Then, by the general Arzel\`{a}-Ascoli theorem, we conclude that $T(\Lambda_M)$ is a relatively compact subset of  $\Theta_R$. 
This completes the proof that $\mathcal{T}$ is compact. 

\noindent\underline{\it Step 1.4: Leray-Schauder fixed point Theorem.} The goal of this step is to apply the Leray-Schauder fixed point Theorem \cite[Theorem 11.3]{gilbarg_trudinger_2001_elliptic}. To that end, we need to prove that the following set is bounded:
\begin{align}\label{set}
    &\bigg\{ (u^X,u^Y) \in \Theta_R\;\left\vert\; 
    \begin{array}{l}
         \exists\xi\in [0,1] \mbox{ such that }  \\
         (u^X,u^Y) =\xi \mathcal{T}(u^X,u^Y) 
    \end{array}
     \bigg\}\,\right..
\end{align}
For this purpose, let $(u^X,u^Y) \in \Theta_R$ be such that
\begin{equation}\label{4.7}
    (u^X,u^Y) =\xi \mathcal{T}(u^X,u^Y)
\end{equation}
for some $\xi\in [0,1]$, and let $(X,Y)$ denote corresponding solution to \eqref{LCBO}. By \eqref{4.7}, the processes $(X, Y)$ are also solutions to
      \begin{align*}
    &\diff X_t
    =-\lambda( X_t-\xi\conx_{\alpha,\beta}(\nu_t^X,\nu_t^Y))\diff t
    +\sigma D( X_t-\xi\conx_{\alpha,\beta}(\nu_t^X,\nu_t^Y)\diff B_t^{X}, \\
    &\diff Y_t
    =-\lambda( Y_t-\xi\cony_{\beta}(\nu_t^Y,P_R(X_t)))\diff t
    +\sigma D( Y_t-\xi\cony_{\beta}(\nu_t^Y,P_R(X_t))\diff B_t^{Y},
\end{align*}
where $\nu_t^X=\mbox{Law} (X_t)$ and $\nu_t^Y=\mbox{Law} (Y_t)$. Then, by Lemma \ref{lem:moment_estimate_mf_cbo} and Remark \ref{remark:add_xi_to_cbo} we find that
\begin{align}
    \label{eq:trunc_mf_cbo_exp_bnd}
    \EE\left[\sup_{t\in[0,T]}|X_t|^p+|Y_t|^p\right]\leq \kappa	\, \big(\EE|X_0|^p+\EE|Y_0|^p\big),
\end{align}
where $\kappa>0$ is a constant. Using the same argument as in \eqref{Ybound}, it follows that $\conx_{\alpha,\beta}(\nu_t^X,\nu_t^Y)$ and $\cony_{\beta}(\nu_t^Y,\cdot)$ can be uniformly bounded in $\Theta_R$. This implies that the set \eqref{set} is indeed bounded. Therefore, we obtain the existence of a fixed point of $\mathcal{T}$, which is a solution to \eqref{eq:mf_cbo1}, and we establish the moment bounds in \eqref{eq:moment_bnd_mf_cbo} as a byproduct. 

These bounds are uniform in $R$, which proves the property \ref{prop:uniform_p_mom_R}. Therefore, by using \eqref{Ybound}, we can show that the set $\{\conx_{\alpha,\beta}(\bX^R,\bY^R)\;\vert\;R>0\}$ is uniformly bounded in $C^0([0,T],\,\RR^{d_1})$, and that the set $\{\cony_{\beta}(\dmf^{Y,R}, \cdot)\;\vert\;R>0\}$ is uniformly bounded in $C^0([0,T]\times B_R,\,\RR^{d_2})$. Therefore, we see that the properties \ref{prop:equic_con_x} and \ref{prop:equic_con_y} hold by \eqref{conditionX} and \eqref{conditionY} respectively.

\smallskip

\underline{\it Step 1.5: Uniqueness of solution.} Let $(u^X,u^Y),(u^{X'},u^{Y'})$ be two fixed points of $\mathcal{T}$
with corresponding solutions $(X,Y)$, $(X',Y'):\Omega\to  \Theta_R$ to \eqref{eq:mf_cbo1}. Define the difference $W_t^m:=m_t-m_t'$, $m\in\{X,Y\}$ and let $\nu_t^m:=\law(m_t)$ for $m\in\{X,Y,X',Y'\}$ . By the BDG inequality and the fact that $\EE|W_0^Y|^p=0$ we have, for all $t\in [0,T]$, 
\begin{align*}
    &\frac{1}{2^{p-1}}\EE\left[\sup_{s\in [0,t]}\; |W_s^Y|^p\right]\\
    &\leq T^{p-1}\lambda^p\int_0^t \EE|W_s^Y-\cony_\beta(\nu_s^Y,P_R(X_s))+\cony_\beta(\nu_s^{Y'},P_R(X_s'))|^p\diff s\\
    &\phantom{\leq} + C_{\mathrm{BDG}}\, T^{p/2-1}\sigma^p\int_0^t \EE|Q|^p\diff s.\\
    &\leq C\left(\int_0^t \EE|W_s^Y|^p\diff s+\int_0^t \EE|\cony_\beta(\nu_s^Y,P_R(X_s))-\cony_\beta(\nu_s^{Y'},P_R(X_s'))|^p\diff s\right),
\end{align*}
where we define the difference
\begin{align*}
    Q:=D\Big( Y_s-\cony_\beta(\nu_s^Y,P_R(X_s))\Big)-D\Big(Y_s'-\cony_\beta(\nu_s^{Y'},P_R(X_s'))\Big),
\end{align*}
and where the constant $C>0$ depends on $T$, $\lambda$, $\sigma$ and $p$. By Corollary \ref{cor:stab_est_cony} and the Lipschitz continuity of $P_R$ we have that 
\begin{align*}
    |\cony_\beta(\nu_r^Y,P_R(X_t))-\cony_\beta(\nu_r^{Y'},P_R(X_t'))|\leq C\big(W_p(\nu_t^Y,\nu_t^{Y'})+|X_t-X_t'|\big).
\end{align*}
Also, by definition of the Wasserstein distance $W_p(\nu_t^Y,\nu_t^{Y'})^p\leq \EE|W_t^Y|^p$. Thus, we conclude
\begin{align*}
    \EE\left[\sup_{r\in [0,t]}\; |W_r^Y|^p\right]
    \leq C\int_0^t  \EE|W_r^X|^p+\EE|W_r^Y|^p\diff r,
\end{align*}
for some constant $C>0$ depending on $T$, $\lambda$, $\sigma$ and $p$. Similarly, we also have
   \begin{align*}
    \EE\left[\sup_{r\in [0,t]}\; |W_r^X|^p\right]
    \leq C\int_0^t  \EE|W_r^X|^p+\EE|W_r^Y|^p\diff r.
\end{align*}
Hence, by the Gronwall inequality we get
\begin{align*}
    \EE\left[\sup_{t\in [0,T]}\; |W_t^X|^p+|W_t^Y|^p\right]
    =0.
\end{align*}
With this we have proven that the solutions to \eqref{eq:mf_cbo1} are well-posed.
\end{proof}

\section*{Acknowledgements}
JW was supported by the Engineering and Physical Sciences Research Council (grant number EP/W524311/1). HH acknowledges the support of the starting grant from Hunan University.
\end{document}